\documentclass[12pt]{amsart}
\usepackage{amsfonts, amssymb, amsmath, amsthm, euscript, amsthm}
\usepackage{graphicx, subfigure, bm}
\usepackage[margin=0.8in]{geometry}
\usepackage{color}

\DeclareMathOperator{\mes}{\bm \mu}
\newcommand{\md}{\mathrm d}
\newcommand{\norm}[1]{\left\lVert#1\right\rVert}
\newcommand{\bx}{\mathbf{x}}
\newcommand{\cE}{\mathcal{E}}
\newcommand{\cL}{\mathcal{L}}

\newcommand{\tZ}{\tilde{Z}}
\newcommand{\tm}{\mathbf{\tilde{m}}}
\newcommand{\cD}{\mathcal{D}}
\newcommand{\cG}{\mathcal{G}}
\newcommand{\E}{\mathbb{E}}
\newcommand{\PP}{\mathbb{P}}
\newcommand{\triplenorm}[1]{{\vert\kern-0.25ex\vert\kern-0.25ex\vert #1
    \vert\kern-0.25ex\vert\kern-0.25ex\vert}}

\DeclareMathOperator{\Uni}{Uni}
\DeclareMathOperator{\TV}{TV}
\DeclareMathOperator{\Exp}{Exp}
\DeclareMathOperator{\Var}{Var}
\DeclareMathOperator{\Poi}{Poi}
\DeclareMathOperator{\Ave}{Ave}

\usepackage[colorlinks,citecolor=blue]{hyperref}

\begin{document}

\newtheorem{thm}{Theorem}
\newtheorem{lemma}[thm]{Lemma}
\newtheorem{cor}[thm]{Corollary}
\theoremstyle{definition}
\newtheorem{asmp}[thm]{Assumption}
\newtheorem{exm}[thm]{Example}
\newtheorem{rmk}[thm]{Remark}
\numberwithin{equation}{section}

\title[Banking System and Financial Contagion]{Dynamic Contagion in a Banking System\\ with Births and Defaults}

\author{Tomoyuki Ichiba}

\address{\scriptsize Department of Statistics and Applied Probability, University of California, Santa Barbara}

\email{ichiba@pstat.ucsb.edu}

\author{Michael Ludkovski}

\address{\scriptsize Department of Statistics and Applied Probability, University of California, Santa Barbara}

\email{ludkovski@pstat.ucsb.edu}

\author{Andrey Sarantsev}

\address{\scriptsize Department of Mathematics and Statistics, University of Nevada, Reno}

\email{asarantsev@unr.edu}

\date{\today. Version 51}



\begin{abstract}
We consider a dynamic model of interconnected banks. New banks can emerge, and existing banks can default, creating a birth-and-death setup. Microscopically, banks evolve as independent geometric Brownian motions.  Systemic effects are captured through default contagion: as one bank defaults, reserves of other banks are reduced by a random proportion. After examining the long-term stability of this system, we investigate mean-field limits as the number of banks tends to infinity. Our main results concern the measure-valued scaling  limit which is governed by a McKean-Vlasov jump-diffusion. The default impact creates a mean-field drift, while the births and defaults introduce jump terms tied to the current distribution of the process. Individual dynamics in the limit is described by the propagation of chaos phenomenon. In certain cases, we explicitly characterize the limiting average reserves.

\end{abstract}

\maketitle

\thispagestyle{empty}

\section{Introduction}
Lending and trading relationships between banks create dependence which can exacerbate financial crises through systemic risk. With this motivation in mind, we study a dynamic model of interacting particles representing the banking network. A particle represents the capital (or net assets) of a financial entity. On the individual level, a particle evolves in time according to a stochastic differential equation, in analogue to classical models of risky assets. On an aggregate or economy-wide level, the particles interact due to inter-bank lending and contractual obligations (such as bilateral derivative claims) that tie the assets and liabilities of different entities, generating mean-field effects.

\smallskip

Focusing on systemic stability, the key aspect of the macroscopic dynamics concerns bank \emph{defaults}. Each particle is viewed as a defaultable asset, meaning it can enter the default state when reserves become low. Financial contagion is then represented through the interaction mechanism which increases default likelihood of other banks once a given bank defaults. Systemic risk emerges as the event of a large number, or cluster, of defaults.

\smallskip

To model such defaults, one may draw upon the two fundamental paradigms in credit risk.

\smallskip

{\it 1. Structural credit models:} Defaults modeled by the first entrance times $\tau  :=\inf\{ t: X_i(t) \in \mathcal{D} \}$, i.e.,~bank capital entering the default region $\mathcal{D}$ (e.g.,~$\mathcal{D} = (-\infty,0]$). In that case, default contagion is usually viewed as a default of bank $i$ affecting the reserves $X_j(\tau)$ of bank $j$, which can generate cascading defaults, i.e.,~multiple banks defaulting simultaneously.

\smallskip

{\it 2. Reduced-form credit models:} Defaults modeled by the \emph{death time} $\tau$ of the particle, captured by a (hazard) rate process that controls the instantaneous probability of default. In this setting, contagion represents heightened default rate of bank $j$ following default of bank $i$, so that defaults \emph{cluster}, but default events are still spaced out in time.

\smallskip

In this work we develop an extension of the interacting particles approach to systemic risk  that makes the financial system dynamic not only on the individual level (bank reserves modeled by stochastic processes), but also in the aggregate (number of banks fluctuates). Thus, we explicitly capture the death (i.e.~default) of existing banks, and the birth of new ones. Indeed, a limitation of existing models is that the size of the system $N$ is either kept constant or is decreasing over time due to defaults. In reality defaulted entities \emph{disappear} and new entities are \emph{created} in analogy to death and birth events in population dynamics. Therefore, aggregate reserves change continuously due to infinitesimal fluctuations in individual reserves, as well as discontinuously due to births/defaults.

\smallskip

Including birth and death of banks carries several important implications. First, it brings the opportunity to obtain stationary models (otherwise the number of active banks will just shrink over time), which is convenient for mathematical analysis, and especially for investigation of \emph{scaling limits}. Stationarity is also necessitated economically for any longer-term model that covers more than a couple of years. Second, our setup offers further contagion mechanisms: We tie individual dynamics both to total system reserves $S(t)$, as well as the number of banks $N(t)$. Third, it brings more realism, paving the way to the next-generation dynamic models and helping to close the gap to the increasingly sophisticated static versions.  Fourth, working with a varying dimension brings nontrivial mathematical challenges in studying the properties of the system, in particular to handle the non-standard state space $\mathcal{X}$ below. To do so, we use McKean-Vlasov jump-diffusions.

\smallskip

To summarize, our main contribution is to describe a class of interacting particle models with a dynamic dimension and mean-field birth and death interactions. Toward this end we: (i) rigorously construct the interacting banking system with local + mean-field default intensities, including investigating its stability; (iii) analyze convergence to a mean field limit for the average bank reserves that leads to a novel jump-diffusion McKean-Vlasov Stochastic Differential Equation (SDE). The drift and diffusion coefficients, as well as the jump measure, of the resulting representative particle depend both on the current position of the process, and the current distribution of the process.

\subsection{Review of existing literature} Systemic risk and financial contagion in financial systems serves as a focus of much recent research, see for instance the handbook \cite{Handbook} describing many different approaches. In the context of a dynamic system with diffusing particles representing bank assets, there are at least three related mean-field approaches. 

\smallskip

Using the reduced-form credit framework, \cite{Giesecke,Cvitanic,Spiliopoulos} modeled the \emph{default rates} $\lambda^n$ of $N$ particles as an interacting diffusion, adding in systemic effects, such as self-exciting defaults and common exogenous shocks. \cite{Bo, Award, Sun} used diffusions interacting through drift to model \emph{bank assets}, with defaults arising structurally from crossing a given default threshold. A related system with interaction through hitting a boundary is discussed in \cite{Kau1}.

\smallskip

In the paper \cite {Fischer}, a mean-field game of interacting particles is introduced, where particles get absorbed upon exiting a certain domain (but there is no emergence of new banks). In the paper \cite{Delarue1} a discrete-space system of interacting particles is used to quantify systemic risk. 

\smallskip

Finally, a nonlocal interaction arising from the default hitting times was recently investigated in the mean-field limit in \cite{Sergey,Sojmark1,Sojmark2, Kau2}. All of the above models either fix the size $N$ of the system, or take $N(t)$ to be non-increasing, representing, say, a fixed pool of defaultable assets that is monitored over time. To our knowledge, the only work that allows $N(t)$ to change have appeared for capturing bank splits/mergers in stochastic portfolio theory~\cite{FouqueStrong,MyOwn4}.



\smallskip

Compared to existing models who tend to focus on short-term (i.e.~a few months to a couple of years), our population-dynamics-inspired setup targets the longer timescale, whereby the concept of a time-stationary banking system becomes appropriate. While there is an ongoing churn among individual banks, our focus is on the macroscopic quantities such as total/mean reserves and number of banks. In line with adoption of the birth-and-death perspective, we focus exclusively on default contagion, eschewing the other mechanisms of systemic dependence, such as interacting drifts or default cascades.

\smallskip

In terms of the mean-field scaling limit, we adapt the results of Graham from the 1990s \cite{Graham,GrahamNew}. Recently several other works investigated mean-field models with particles undergoing jump diffusions. In particular, a growing strand of literature \cite{Delarue1,Delarue2,neuron-new,toy,Mehri18} investigates neuronal networks where $X_i(t)$ are electrical states of individual neurons. These models feature jump diffusions that capture spikes from neurons firing, however the mean-field interaction is limited to the drift and jump size terms; jump activity is taken to be a Poisson process with a deterministic local intensity. 

\smallskip

An extension to simultaneous jumps which transform to a drift term in the limit and are similar to our contagion mechanism appears in \cite{Andreis}. While the above works also establish the hydrodynamic McKean-Vlasov limit existence and propagation of chaos, their pre-limit models always feature a constant number of particles $N$ so the scaling procedure of $N \to \infty$ is standard. In contrast, endogenizing $N$ creates multiple scaling alternatives which is one of the main foci of our work. Finally, we should also mention \cite{Campi17,Campi18} who analyzed mean-field \emph{games} with jump-diffusions, however again they only consider interaction in the jump sizes.

\subsection{Informal description of the model} \label{sec: Informal} We model the financial system by a vector of continuous time stochastic processes (individual ``particle'' locations) $X_i = (X_i(t),\, t \ge 0)$, with $X_i(t) \ge 0$ standing for the reserves of the corresponding bank $i$ at time $t \ge 0$. Low $X_i(t)$ means that the bank has minimal reserves and is close to being financially insolvent; healthy banks should have large reserves. Let $I(t) \in 2^{\mathbb{N}}$ be the finite set of banks at time $t \ge 0$ and
$$
N(t) = |I(t)|,\quad S(t) := \sum\limits_{i \in I(t)}X_i(t),
$$
so that  $N(t)$ is the number of banks and $S(t)$ is the sum of their reserves at time $t$. Locally, each $X_i$ behaves as an independent geometric Brownian motion, representing the idiosyncratic shocks to the reserves of the $i$th bank. Banks randomly emerge and default. Birth of new banks has time-varying intensity $\lambda_\cdot$ and starting size distribution $\mathcal{B}$, both depending on $N(t)$ and $S(t)$. The respective dependence captures the idea that forming a new bank is easier with less  competition.

\smallskip

An existing bank $i$ {\it defaults} ($X_i$ is killed) with intensity $\kappa_t$ depending on $N(t)$, $X_i(t)$, and $S(t)$. Default becomes more likely as $X_i$ drops; safety from default requires larger reserves. A default by $i$ affects other banks $j \neq i$: Their reserves $X_j(t)$ decrease at the default epoch by a random factor $\xi_{ij}$, which is dependent on $N(t)$, $X_i(t)$, $S(t)$, and idiosyncratic factors related to these particular banks $i$ and $j$. This models {\it financial contagion} in the interconnected financial system, including the intuition that defaults of larger banks $X_i(\tau-)$ trigger more contagion than smaller ones. 
%
The overall rules governing the system dynamics are thus:

\smallskip

(a) As long as the number of banks stays constant, each of them behaves as a geometric Brownian motion with drift $r$ and volatility $\sigma$, independently of other banks.

\smallskip

(b) A new bank is added to the system with rate $\lambda_{N(t)}(S(t))$. This bank has initial reserves distributed according to a probability measure $\mathcal B_{n, s}$ on $(0, \infty)$ when $N(t-) = n$, $S(t-) = s$.
When $n=0$, we write $\mathcal B_{n, s} = \mathcal B_0$ for all $s > 0$; this governs the distribution of the new bank reserves when it is the first emerging bank. We denote by
$\overline{\mathcal B}(n, s)$ the mean size of a new bank, i.e.~the first moment of $\mathcal B_{n,s}$.

\smallskip

(c) An existing bank $i \in I(t)$ {\it defaults} with rate $\kappa_{N(t)}(S(t),X_i(t))$. At the moment of default, reserves of remaining banks $j \in I(t),\, j \ne i$, decrease by a fraction 
$$
\xi_{ji} \sim \cD_{N(t), S(t), X_i(t)}
$$
which are i.i.d.~random variables with values in $(0, 1)$. The measure $\cD_{n,s,x'}$, with mean  $\overline{\cD}(n, s,x)$, governs the proportional impact of default given the number of banks, their total reserves, and the size of the defaulting bank $x'$.

\smallskip

\subsection{Questions of interest} First, we investigate conditions on this system to be well-defined probabilistically. In particular, we establish conditions for the system to be {\it conservative}: defined on the infinite time horizon. Next, we study the stronger notion of {\it stability} of this system: Whether the vector of $X_i(t)$ converges to some limiting distribution as $t \to \infty$. To find sufficient conditions for stability we use two different methods: (a) Lyapunov functions, developed in classic papers \cite{MT1993a, MT1993b}; (b) comparison of $\{N(t)\}$ with a birth-death process.

\smallskip

Our main analysis is devoted to the limiting behavior of this system as the number of banks tends to infinity. After the proper scaling of birth and default intensities, the empirical distribution of $X_i(t)$ converges to a measure-valued process, which is a solution to a certain McKean-Vlasov stochastic differential equation with jumps, i.e.,~a nonlinear diffusion with discrete jump sets. For this process, the drift and diffusion coefficients, as well as the jump measure, depend not only on the current location of the process (as would be for a classical jump-diffusion), but also on the current {\it distribution} of this process. This is a {mean field limit.}

\smallskip

In fact, we find two different mean-field limits, with parameters scaled: (a) according to the {\it current} number of banks; (b) according to the {\it initial} number of banks. In both cases, the limit is a McKean-Vlasov jump-diffusion, but in case (b), the parameters (drift and diffusion coefficients, jump measures) depend on the whole history, rather than on the current state and distribution, of the process. Both limits are financially viable, depending on the birth-and-death rates $\lambda, \kappa$.

\smallskip

In certain cases, the McKean-Vlasov equation turns out to allow an explicit solution: geometric Brownian motion with time-dependent drift, killed with certain rate and then resurrected at a certain given probability distribution. Financial contagion described above leads to an additional drift coefficient in the limit, while emergence of banks creates the phenomenon of resurrection. Economically, this limit offers an equilibrium justification for using a local-intensity defaultable geometric Brownian motion model for an individual risky asset. Furthermore, we show that the time-stationary version of this limiting process is a mixture of lognormal distributions. 

\smallskip

\emph{Systemic risk} corresponds to a large number of defaults in our system. This can be interpreted as an event in terms of $N(T)$ for some horizon $T$, or a joint event about $\{ N(T), S(T)\}$. Probabilities of such events can be evaluated numerically with our model; the mean field limit offers additional insights into the distribution of the mean bank size.

\smallskip

Lastly, we examine the behavior of an individual bank under these limits. It converges to a diffusion process, similar to geometric Brownian motion, with constant diffusion coefficient and an (easily computable) time-dependent drift, killed at a certain rate. For two banks (or any finite number), dependence vanishes in the limit. The corresponding processes converge to independent copies of such processes, similar to geometric Brownian motions. This phenomenon is called {\it propagation of chaos}. 

\subsection{Organization of the paper} In Section~\ref{sec:construction}, we introduce necessary notation, and construct our model formally. In Section~\ref{sec:stability}, we find sufficient conditions for no explosions and for stability of this system, as well as estimating rate of convergence. In Section \ref{sec:mean-field}, we consider large systems, to obtain the (first) mean field limit (scaling by current number of banks) and the resulting McKean-Vlasov-It\^o-Skorohod process. We apply this to systemic risk. Finally, we consider behavior of individual banks in these large systems.  In Section~\ref{sec:mf-N}, we establish the second result (scaling by the initial number of banks). Sections \ref{app:technical}-\ref{sec:theorem42} are devoted to proofs.
 Appendix in Section \ref{sec: Appendix} collects auxiliary results. 

\section{Definitions and Formal Description}\label{sec:construction}

\subsection{Notation} Before constructing the system, let us define the state space
$$
\mathcal X := \bigcup\limits_{N=0}^{\infty}(0, \infty)^N,
$$
with the understanding that $(0, \infty)^0 := \{\varnothing\}$, corresponding to the case of no banks (empty banking system). This is a Hausdorff topological space with disconnected components $(0, \infty)^N,\, N = 0, 1, 2, \ldots$. We define the Lebesgue measure $\mes$ on $\mathcal X$, which coincides with the $N$-dimensional Lebesgue measure on $(0, \infty)^N$ for each $N \ge 1$, and $\mes(\{\varnothing\}) = 1$ for $N = 0$. We denote the integral of a measurable function $f : (0, \infty) \to \mathbb R$ with respect to a probability measure $\nu$ as $(\nu, f) := \int f(x) \,\nu ({\mathrm d}x) $. Let $f_p (x) := x^p$ for $p, x > 0$. For each 
$\mathbf{x} = (x_1, \ldots, x_{\mathfrak n(\bx)}) \in \mathcal X$, define:

\smallskip

\noindent (a) the  dimension $\mathfrak n(\bx)$ of $\bx$  , i.e.,~ if $\mathbf{x} \in (0, \infty)^{N}$, then $\mathfrak n(\bx) = N$ for $N  = 1, 2, \ldots$, and $ \mathfrak n (\varnothing) = 0$;

\smallskip

\noindent  (b) the sum $\mathfrak{s}(\bx) := \sum_{k=1}^{\mathfrak n(\bx)}x_k$, with $\mathfrak{s}(\varnothing) := 0$;

\smallskip

\noindent  (c) the empirical measure corresponding to $\bx (\neq \varnothing) $ :
\begin{equation} \label{eq: empiricaldist}
\mu_{\mathbf{x}} := \frac1{\mathfrak n(\bx)}\sum\limits_{i=1}^{\mathfrak n(\bx)}\delta_{x_i}(\cdot) \, \text{ and } \mu_{\varnothing} := \delta_{\varnothing} ;
\end{equation}

\smallskip

\noindent (d) for any function $f : (0, \infty) \to \mathbb R$, a corresponding function $\mathcal E_f: \mathcal X \to \mathbb R$:
\begin{equation}
\label{eq:empirical-function}
\mathcal E_f(\bx) := (\mu_{\mathbf{x}}, f) = \frac1{\mathfrak n(\bx)}\sum\limits_{i=1}^{\mathfrak n(\bx)}f(x_i);
\end{equation}

\noindent (e) the mean (average) $\overline{\bx} = \mathfrak{s}(\bx)/\mathfrak{n}(\bx) = (\mu_{\bx}, f_1) = \mathcal E_{f_1}(\bx)$.

\smallskip

A subset $E \subseteq \mathcal X$ is compact if it intersects only finitely many levels $(0, \infty)^N$, and if the intersection with each such level is compact in the usual Euclidean topology. Denote by $\mathcal P_p$ the family of all probability measures on $\mathbb R_+$ with finite $p$-th moment. This is a metric space under the {\it Wasserstein distance:}
\begin{equation}
\label{eq:Wasserstein}
\mathcal W_p(\nu', \nu'') = \inf_{(\xi', \xi'')} \E [ |\xi' - \xi''|^p]^{1/p}, \qquad \nu', \nu'' \in \mathcal P_{p} ,
\end{equation}
where the {infimum} in~\eqref{eq:Wasserstein} is taken over all couplings $(\xi', \xi'')$ of random variables with marginals $\nu', \nu''$, respectively from the family $\mathcal P_{p}$ for $ p \ge 1$. For $p \in (0, 1)$, the distance~\eqref{eq:Wasserstein} is not a metric, but it generates a topology. It is known that convergence in this space is equivalent to the weak convergence plus convergence of the $p$th moments. Here weak convergence of probability measures or random variables is denoted by $\Rightarrow$.

A {\it geometric Brownian motion} with {\it drift} $\mu$ and {\it diffusion} $\sigma^2$ is defined as
$$
x_0\exp\left(\mu t + \sigma W(t)\right),\quad t \ge 0
$$
 for a Brownian motion $W$ on a filtered probability space and starting point $x_0$. We assume that all banks share fixed volatility $\sigma$ and drift $\mu := r-\sigma^2/2$ where $r \ge 0$ is the {asset growth rate}. (Those quantities could be also  straightforwardly randomized, in an i.i.d.~manner across the banks.) Let $\mathcal C$, $\mathcal C_b$, and $\mathcal C^2$  be the spaces of continuous, bounded continuous, and twice continuously differentiable functions $(0, \infty) \to \mathbb R$, respectively. For bounded functions $f : (0, \infty) \to \mathbb R$, define $ \norm{f} := \sup_{x  > 0}|f(x)|$. Define the following operators on $\mathcal C^2$:
\begin{equation}
\label{eq:operators}
D_1f(x) := xf'(x),\qquad D_2f(x) := x^2f''(x),\qquad \cG f := rD_1 f + \frac{\sigma^2}2D_2 f,
\end{equation}
so that $\cG$ is the infinitesimal generator of a geometric Brownian motion. Operators in~\eqref{eq:operators} preserve the monomial function $f_p$ up to a constant multiple:
\begin{equation}
\label{eq:acting-simply}
D_1f_p = pf_p,\quad D_2f_p = p(p-1)f_p,\quad \cG f_p = \bigl(\sigma^2p(p-1)/2 + pr\bigr)f_p.
\end{equation}
For a measure $\nu \in \mathcal P_1$, we denote its mean by $\overline{\nu} := (\nu, f_1)$. In particular, $\overline{\nu_{\bx}} = \overline{\bx}$. Define the space
$$
\mathcal C^2_b := \{f \in \mathcal C^2\mid f, D_1f, D_2f \in \mathcal C_b\}
$$
with the norm which makes it a Banach space:
\begin{equation}
\label{eq:triplenorm}
\triplenorm{f} := \norm{f} + \norm{D_1f} + \norm{D_2f}.
\end{equation}
The {\it total variation distance} between two probability measures $P$ and $Q$ on $\mathcal X$:
\begin{equation}
\label{eq:TV}
\norm{P - Q}_{\TV} = \sup\limits_{f : \mathcal X \to \mathbb R,\, |f| \le 1}|(P, f) - (Q, f)| = 2\sup\limits_{A \subseteq \mathcal X}
|P(A) - Q(A)|.
\end{equation}
A generalization of~\eqref{eq:TV} is defined as follows: Fix a function $V : \mathcal X \to [1, \infty)$, and let
\begin{equation}
\label{eq:V-norm}
\norm{P - Q}_V := \sup\limits_{f : \mathcal X \to \mathbb R, |f| \le V}|(P, f) - (Q, f)|.
\end{equation}
For $V \equiv 1$, the norm~\eqref{eq:V-norm} becomes the usual total variation norm from~\eqref{eq:TV}. Convergence in such norms is in some sense stronger than weak convergence or convergence in Wasserstein distance: The former requires convergence for {\it all measurable} test functions (bounded by a constant or by a function, depending on the measure), while the latter does   only for {\it continuous} test functions.

\smallskip

Finally, for a metric space $(E, \rho)$, define the Skorohod space $\cD([0, T], E)$ of $E$-valued, right-continuous functions with left limits ({\it rcll}) on $[0, T]$. In particular, $\cD[0, T] := \cD([0, T], \mathbb R)$.

\subsection{Formal description of the system} Take a filtered probability space $(\Omega, \mathfrak F, (\mathfrak F (t)), \PP)$ endowed with the following independent random objects:

\smallskip

(a) an initial condition $\mathbf{x}_0 \in \mathcal X$;

\smallskip

(b) i.i.d.~random variables $\xi_{i,j,n, s, x} \sim \cD_{n, s,x}$ for $s > 0$, $0 < x < ns$, $i, j = 1, 2, \ldots$;
\smallskip

(c) i.i.d.~Brownian motions $W_{i,j}$ for $i = 0, 1, 2, \ldots$ and $j = 1, 2, \ldots$

\smallskip

(d) i.i.d.~$(0, \infty)$-valued random variables $\zeta_{k, n, s} \sim \mathcal B_{n, s}$ for every $k, n = 0, 1, 2, \ldots$ and $s > 0$;

\smallskip

(e) i.i.d.~exponential random variables $\eta_{k, i}$ with mean $1$ for $i, k = 0, 1, 2, \ldots$,

\noindent where $\cD_{n, s, x}$ is a probability distribution on $(0, 1)$ with average default impact $ \overline{\cD}$, depending on $n, x, s$, and $\mathcal B_{n,s}$ is a probability distribution on $(0, \infty)$ with average size  $ \overline{\mathcal B}$ of new bank, depending on $n, s$.  See Section \ref{sec: Informal} for informal description.
\smallskip

Our model consists of three components $(X, I, M)$:

\smallskip

(A) an $\mathcal X$-valued continuous-time process $X := (X(t),\, t \ge 0)$ with right continuous with left limits (r.c.l.l.) trajectories, which jumps at random times $0 = \tau_0  < \tau_1 < \tau_2 < \ldots$, and on each time interval $[\tau_k, \tau_{k+1}),\, k = 0, 1, 2, \ldots$ has constant dimension $N(t) := \mathfrak n( X(t))$;

\smallskip

(B) a set-valued process $I := (I(t),\, t \ge 0)$ such that for $t \ge 0$, $I(t) \in 2^{\mathbb{N}}$ is a finite set of positive integers, which is constant on each time interval $[\tau_k, \tau_{k+1})$ for $k = 0, 1, 2, \ldots$, with $|I(t)| = N(t)$; this is the set of the names of current banks. Initially, $I(0) = \{ 1,2,\ldots, \mathfrak{n}( \bx_0)\}$.

\smallskip

(C) a nondecreasing positive integer-valued process $M := (M(t),\, t \ge 0)$, which is also constant on each interval $[\tau_k, \tau_{k+1})$, such that $M(t) := \max \{ k : k \in \cup_{s \in [0, t]}I(s)\}$; this is the maximum index or name of a bank which existed so far at some point.

\smallskip

We define $ (X, I, M) $ inductively with $\, \lvert I ( \cdot ) \rvert =  \mathfrak n ( X(\cdot)) \le M(\cdot)\,$ on the time interval $[0, \tau_{\infty})$, where $\tau_{\infty} := \lim_{k \to \infty}\tau_k$. The detailed, formal construction is discussed in Appendix.  By construction, this is a Markov process on the state space
\begin{equation}
\Xi := \{(\mathbf{x}, \mathfrak{i}, \mathfrak m) \in \mathcal X\times 2^{\mathbb N}\times \mathbb N \, : \,  |\mathfrak{i}| = \mathfrak n(\bx) \le \mathfrak m\} ,
\end{equation}
and its law is uniquely determined up to explosion time. The  generator  $\mathfrak{L}$ 
of $(X, I, M)$ is given by 
\begin{align}
\label{eq:generator}
\begin{split}
&  \mathfrak{L}f (\mathbf{x}, \mathfrak{i}, \mathfrak m)  = \sum\limits_{i \in \mathfrak{i}}\left(rx_i\frac{\partial f}{\partial x_i} + \frac12\sigma^2x_i^2\frac{\partial^2f}{\partial x_i^2}\right) \\ & {} + \lambda_{\mathfrak n(\bx)}(\mathfrak{s}(\bx))\int_0^{\infty}\left[f((\mathbf{x}, y), \mathfrak{i}\cup\{\mathfrak m+1\}, \mathfrak m + 1) - f( \mathbf x, \mathfrak{i}, \mathfrak m)\right]\mathcal B_{\mathfrak n(\bx), \mathfrak{s}(\bx)}(\md y) \\ & {} + \sum\limits_{i \in \mathfrak{i}}\kappa_{\mathfrak n(\bx)}
(\mathfrak{s}(\bx), x_i)\!\!\!\!\!\int\limits_{(0, 1)^{\mathfrak{n}(\mathbf{x})-1}}\!\!\!\!\left[f\left({\mathbf{x}}_{-i}\! \circ\! (\mathbf{e} - \mathbf{z}_{-i}), \mathfrak{i}\setminus\{i\}, \mathfrak m\right) - f(\mathbf{x}, \mathfrak{i}, \mathfrak m)\right]\cD_{\mathfrak n(\bx), x_i, \mathfrak{s}(\bx)}^{\otimes\mathfrak{n}(\mathbf{x})-1}(\md {\mathbf z}_{-i}),
\end{split}
\end{align}
where we denote by ${\mathbf{x}}_{-i}$ any vector $\mathbf{x} \in \mathcal X$ with its $i$-th component $x_i$ removed; $\mathbf{e}$ is the vector of units of size $\mathfrak n (x) - 1 $;
$\mathbf{z}_{-i}$ is a vector in $(0, 1)^{\mathfrak{n}(\mathbf{x}) - 1}$, $\circ$ is used for the Schur product, i.e., element-wise multiplication of vectors, and
$\mathcal Q^{\otimes m}$ is the direct product of $m$ copies of a probability measure $\mathcal Q$. The three terms on the different lines of \eqref{eq:generator} represent the continuous diffusion, births, and defaults of the banks, respectively.
The domain of $\mathfrak{L}$ in ~\eqref{eq:generator} is the space of functions $f : \Xi \to \mathbb R$ such that for every
$(\mathfrak{i}, \mathfrak m)$ the restriction $\mathbf{x} \mapsto f(\mathbf{x}, \mathfrak{i}, \mathfrak m)$ belongs to the space
$$
\mathcal C^2(\mathcal X) := \{f : \mathcal X \to \mathbb R \, :\,  \left.f\right|_{(0, \infty)^N} \in C^2((0, \infty)^N),\ \forall N = 1, 2, \ldots\}.
$$
A sum over the empty set is understood to be zero. Sometimes, abusing the notation slightly, we shall apply $\mathfrak{L}$ to a function $f : \mathcal X \to \mathbb R$, $f \in \mathcal C^2(\mathcal X)$, and regard $\mathfrak{L} f$ as a function only on $\mathcal X$, in effect ignoring auxiliary variables and concentrating only on the state space $\mathcal X$.

\smallskip

A sample path of $X$, together with the corresponding $S(\cdot) = \sum_{i \in I(\cdot)} X_{i}$ and $ N(\cdot)$, is shown in Figure \ref{fig: 1}. One can clearly observe the contagion mechanism: as one bank defaults, the other reserves also drop, which due to the increased $\kappa$ (default rate being hyperbolic in available reserves) is likely to trigger further defaults. Consequently, there is a self-excitation effect to the downward jumps of $N(t)$ (while the upward jumps corresponding to births have a constant rate $\lambda$).

\begin{figure}[t]
\begin{center}
\includegraphics[scale=0.4, trim=0.2in 0.2in 0.2in 0.2in]{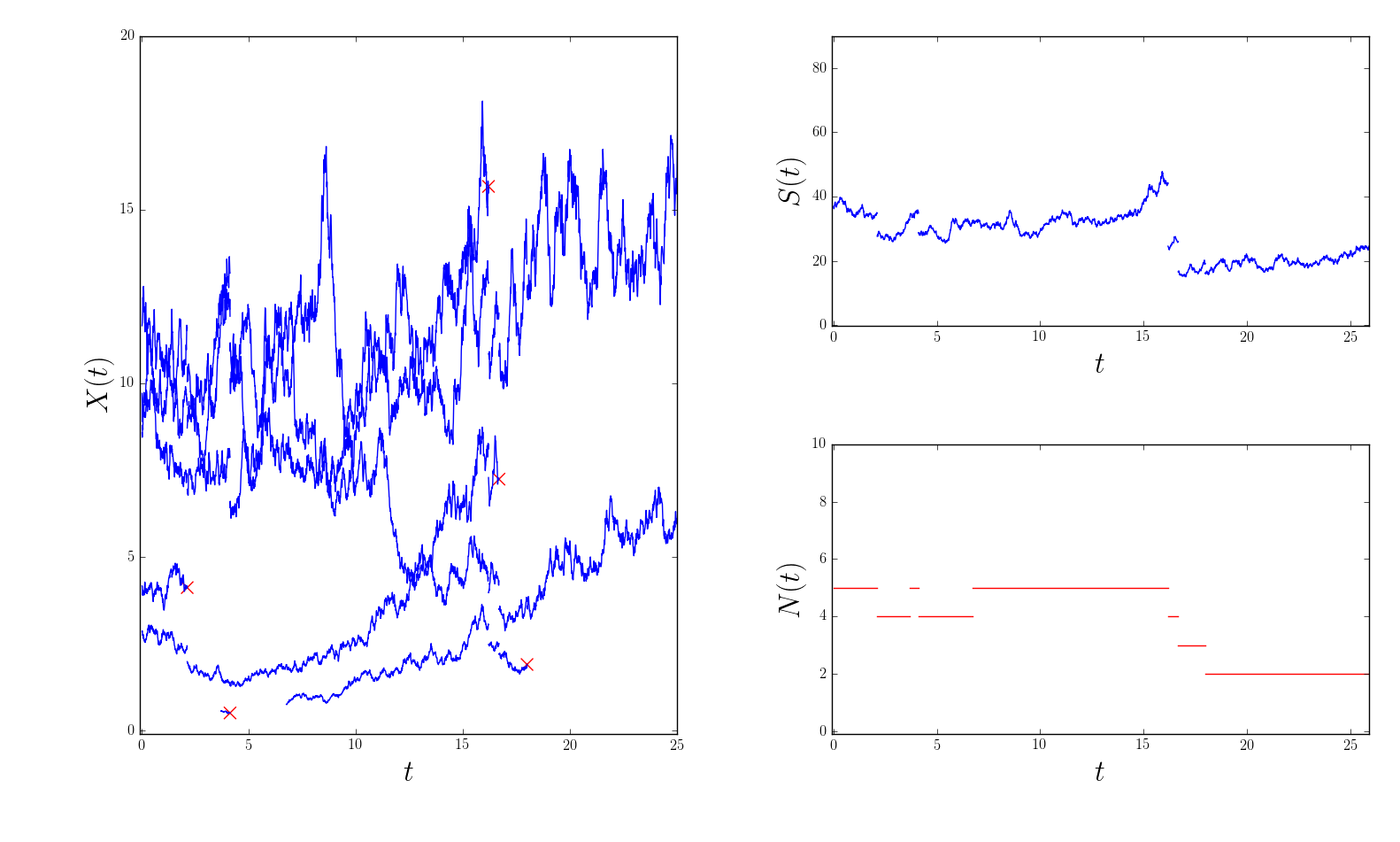}
\end{center}
\caption{The sample paths of $X, S, N$ with $N(0) := 5$, $\sigma := 0.2$, $r := 0.05$. The initial distribution of $\, X\,$ is i.i.d. standard log normal. Default rate is hyperbolic in bank reserves $x$: $\kappa_n(s,x) = 0.2\cdot n / (0.01+x)$. The contagion measure (default impact) is uniform, i.e., $\mathcal D_{n,s,x} \sim \Uni[0, n^{-1}]$, the birth rate is constant $\lambda_n(s) = 1$,  and the new bank distribution is $\mathcal B_{n,s} \sim \Exp(1)$ for every $n$ and $s$. The $\times$ markers in the left panel represent defaults. } \label{fig: 1}
\end{figure}

\section{Existence and Stability}\label{sec:stability}


\subsection{Conditions for existence}

The following two lemmas describe the elementary properties of the Markov transition kernel of $X$. First, the process $X$ is {\it totally irreducible}. That is, its transition kernel $P^t(\mathbf{x}, \cdot)$ is positive with respect to the Lebesgue measure $\mes$ on $\mathcal X$.
\begin{lemma} \label{lemma:totally-irreducible} For all Borel subsets $A \subseteq \mathcal X$ with $\mes(A) > 0$, we have:
\begin{equation}
\label{eq:totally-irreducible}
P^t(\mathbf{x}, A)  := \mathbb P ( X(t) \in A \,  \vert \, X(0) = \mathbf x)   > 0\quad \mbox{for all}\quad t > 0,\, \mathbf{x} \in \mathcal X.
\end{equation}
\end{lemma}
\noindent This {\it positivity property} is used 
for the stability of $X$. The proof of Lemma~\ref{lemma:totally-irreducible} is  in
Section \ref{sec:Lemma31}.


\smallskip

Second, due to local boundedness of the intensities of birth and default, $X$ satisfies the Feller property.  The following Lemma follows from construction of $X$ by {\it patching}: constructing the continuous parts jump-after-jump; see \cite{Bass1979, Sawyer}.
\begin{lemma}
\label{lemma:Feller}
The process $X$ is Feller continuous: That is, for any bounded continuous function $f : \mathcal X \to \mathbb R$, with convention that $f(\Delta) = 0$, where $\Delta$ is the (isolated) cemetery state, the function $P^tf$ is also bounded and continuous for every $t > 0$.
\end{lemma}

We proceed to state some sufficient conditions when $X$ is {\it conservative},  i.e.,~well-defined on the infinite time horizon so that  $\tau_{\infty} = \infty$ a.s. We sometimes say in this case that the system {\it does not explode}. To this end, it suffices to find a {\it Lyapunov function}. This is a standard tool to prove that a random process is conservative or stable: See for example classic papers \cite{MT1993a, MT1993b, DMT1995}. Essentially, for our proof that $X$ is conservative, we need a function $V : \mathcal X \to [0, \infty)$ such that:

\smallskip

(a) for every $c > 0$ the set $\{\mathbf{x} \in \mathcal X\mid V(\bx) \le c\}$ is compact (informally ``$V(\infty) = \infty$'');

\smallskip

(b) $V \in C^2(\mathcal X)$, and for some constants $k, c > 0$, $\mathfrak{L} V(\bx) \le kV(\bx) + c$ for all  $\mathbf{x} \in \mathcal X$.

\smallskip

For our setting, let us take the following Lyapunov function:
\begin{equation}
\label{eq:Lyapunov-example}
V_0(\bx) := \mathfrak{s}(\bx) + \mathfrak n(\bx), \quad \mathbf x \in \mathcal X .
\end{equation}
This function trivially belongs to $C^2(\mathcal X)$. By construction of the topology on $\mathcal X$, the function $V_{0}$ from~\eqref{eq:Lyapunov-example} satisfies the property (a) above. Plugging ~\eqref{eq:Lyapunov-example} in ~\eqref{eq:generator}, we get after calculations (recall that $\overline{\mathcal B}$ and $\overline{\mathcal D}$ are the means of the respective distributions):
\begin{align}
\label{eq:Lyapunov-result}
\begin{split}
\varphi(\bx)  & := \mathfrak{L}  V_{0}(\bx)   =  r\mathfrak{s}(\bx)  + \lambda_{\mathfrak n(\bx)}(\mathfrak{s}(\bx))\left[\overline{\mathcal B}(\mathfrak n(\bx), \mathfrak{s}(\bx)) + 1\right] \\ & \qquad  \qquad  \qquad {} - \sum_{j =1}^{\mathfrak n(\bx)}\kappa_{\mathfrak n(\bx)}(\mathfrak{s}(\bx), x_j)\left[\overline{\cD}(\mathfrak n(\bx), \mathfrak{s}(\bx), x_i)\mathfrak{s}(\bx) + 1\right]  \\ & \qquad \qquad \qquad  {} - \sum_{i=1}^{\mathfrak n(\mathbf x)}x_i\kappa_{\mathfrak n(\bx)}(\mathfrak{s}(\bx), x_i)\left(1 - \overline{\cD}(\mathfrak n(\bx), \mathfrak{s}(\bx), x_i)\right).
\end{split}
\end{align}
Under some assumption on this function $\varphi(\cdot) $, we claim the following. Its proof is in Section \ref{sec:proofofconserva}.
\begin{thm} Assume there exist positive constants $c_1, c_2, c_3$ such that $\varphi(\cdot) $ in \eqref{eq:Lyapunov-result} satisfies
$$
\varphi(\bx) \le c_1\mathfrak{s}(\bx) + c_2\mathfrak n(\bx) + c_3\quad\mbox{for}\quad \mathbf{x} \in \mathcal X.
$$
Then the system exists and is conservative: It does not explode.
\label{thm:conservative}
\end{thm}

\begin{exm} The simplest conservative example  is described when the birth $\lambda$ and default rates $\kappa$ are independent of $N(t)$ and $S(t)$. Then the number $N(t)$ of banks at time $t$ forms a birth-death process with birth intensity $\lambda$ and death intensity $\kappa N(t)$. Hence, we can apply the usual sufficient conditions for this process being conservative. If this process is conservative, then the whole system is also conservative, since on each level $\{\mathfrak n(\bx) = N\}$ (for a given number $N$ of banks), the system behaves as a collection of independent geometric Brownian motions.
\end{exm}



\subsection{Stability of the system} On a macro-level, to obtain stability we need some balance between births and defaults. Recall that A probability measure $\Pi$ on $\mathcal X$ is called a {\it stationary distribution} or {\it invariant measure} for the system above if the following holds: If we start $X(0) \sim \Pi$, then for all $t \ge 0$, we remain at $X(t) \sim \Pi$. The system is called {\it stable} if it is nonexplosive, there exists a unique stationary distribution $\Pi$, and for every given initial condition $X(0) \in \mathcal X$, the distribution of $X(t)$ converges to $\Pi$ as $t \to \infty$ in the total variation distance:
\begin{equation}
\label{eq:ergodic}
\lim\limits_{t \to \infty}\sup\limits_{A \subseteq \mathcal X}\left|\PP(X(t) \in A) - \Pi(A)\right| = 0.
\end{equation}

\begin{thm} \label{thm:stable}
The system is stable if the set $\{\mathbf{x} \in \mathcal X\mid \varphi(\bx) > -\varepsilon\}$ is compact for some $\varepsilon > 0$, i.e., the function $\varphi$ from~\eqref{eq:Lyapunov-result} satisfies $\varlimsup_{\mathbf{x} \to \infty}\varphi(\bx) < 0$. 
\end{thm}
This result immediately follows from \cite{MT1993a, MT1993b} and Lemmata~\ref{lemma:totally-irreducible}, \ref{lemma:Feller}, \cite[Proposition 2.2, Lemma 2.3]{MyOwn10}.
From the definition of compactness in $\mathcal X$ from Section 2, there exist constants $\varphi_0, s_0, N_0 >0$ such that $\varphi(\bx) \le -\varphi_0$ for $\mathfrak{s}(\bx) \ge s_0$ or $\mathfrak n(\bx) \ge N_0$.

\begin{exm} A simple condition for stability is to have banks with finite lifetime, i.e.,~the default time $\tau_i$ of any bank $i$ is finite a.s. In that  case the system will be stable as long as the birth rate remains bounded.  Assume $\mathcal B_{n, s}$ and $\cD_{n, s,x}$ depend only on $n$, and
$$
\lambda_* := \sup\limits_{n}\frac1n\lambda_n(s) < \infty,\ \mbox{and}\ \kappa_n(s,x_i) \equiv g(x_i),
$$
for some decreasing function $g : (0, \infty) \to \mathbb R_+$, with $g(0+) = +\infty$ and $g(+\infty) =: \kappa_* > 0$. Then each bank has finite lifetime, which is dominated from above by an exponential random variable with rate $\kappa_*$. Therefore, the quantity of banks is stochastically dominated by a birth-death process with birth intensities $\lambda_* n$ and death intensities $\kappa_*n$ at level $n \ge 1$. If $\kappa_* > \lambda_*$ then this birth-death process is stable. Combining this observation with the independence of $\mathcal B_{n, s}$ and $\mathcal D_{n, s, x}$ of $x$ and $s$, we get that the whole system $X$ is stable.
\end{exm}

\subsection{Refinements of stability results}
A stronger convergence than~\eqref{eq:ergodic} (in the {\it total variation distance} from~\eqref{eq:TV}) can happen exponentially fast as $t$ grows: There exist positive constants $C$ (depending on the initial condition $X(0)$) and $\alpha$ such that
\begin{equation}
\label{eq:exp-ergodic}
\sup\limits_{A \subseteq \mathcal X}\left|\PP(X(t) \in A) - \Pi(A)\right| \le Ce^{-\alpha t}.
\end{equation}

\begin{thm} \label{thm:stability}
The system $X$ satisfies~\eqref{eq:exp-ergodic} if there exist constants $c_1, c_2, c_3 > 0$ such that
\begin{equation}
\label{eq:1205}
\varphi(\bx) \le -c_1\mathfrak s(\mathbf x) - c_2\mathfrak n(\mathbf x) + c_3,\quad \mathbf x \in \mathcal X,
\end{equation}
where $\varphi(\bx)$ is defined in \eqref{eq:Lyapunov-result}. More generally, the system satisfies a stronger convergence statement:  For the function $V_0$ defined in~\eqref{eq:Lyapunov-example}, there exist positive constants $C$ and $\alpha$ such that for all $\bx \in \mathcal X$, the transition function $P^t(\bx, \cdot)$ of the Markov process $X$ satisfies
\begin{equation}
\label{eq:exp-ergodic-V}
\norm{P^t(\bx, \cdot) - \Pi(\cdot)}_{V_0} \le CV_0(\bx)e^{-\alpha t}.
\end{equation}
\end{thm}

\begin{exm}  Assume the following parameters do not depend on $n$, $x$, and $s$:
$$
\lambda_n(s) \equiv \lambda,\quad \kappa_n(s,x) \equiv \kappa,\quad \overline{\mathcal B}(n, s) \equiv \overline{\mathcal B},\quad \overline{\cD}(n, s, x) \equiv \overline{\cD}
$$
with some constants $\lambda ,\kappa ,\overline{\mathcal B}, \overline{\cD} > 0$. Then  the function $\varphi$ from~\eqref{eq:Lyapunov-result} becomes
$$
\varphi(\bx) = r\mathfrak{s}(\bx) + \left[\overline{\mathcal B} +1\right]\lambda - \overline{\cD} \kappa\cdot\mathfrak{s}(\bx)\mathfrak{n}(\bx) - \kappa \mathfrak{n} (\bx) - \kappa(1 - \overline{\cD})\cdot\mathfrak{s}(\bx).
$$
Since $\mathfrak n(\bx) \ge 1$ for $\mathbf{x} \in \mathcal X\setminus\{0\}$, the condition of Theorem~\ref{thm:stable} holds when $\kappa(1 - \overline{\cD}) > r$; that is, when the intensity of defaults, adjusted by the average contagion effect exceeds the growth rate of non-defaulting bank reserves.
\end{exm}

Finally, we can sometimes find an {\it explicit} estimate for the rate $\alpha$ of exponential convergence. This is done using the {\it coupling} argument from \cite{LMT1996, MyOwn12, MyOwn16}.

\begin{thm} \label{thm:explicit-rate} Assume $\lambda_n(y) \le \lambda^*_n$ and $\kappa_n(y, x) \ge \kappa^*_n$ for all $n, x, y$. Take a nondecreasing function $\hat{V} : \{0, 1, 2, \ldots\} \to [1, \infty)$ such that $\hat{V}(0) = 1$, and
\begin{equation}
\label{eq:exact-L}
\lambda^*_n\hat{V}(n+1) + n\kappa^*_n\hat{V}(n-1) - (\lambda^*_n + n\kappa^*_n)\hat{V}(n) \le -\alpha V(n),\quad n = 1, 2, \ldots
\end{equation}
Define $\tilde{V} : \mathcal X \to [1, \infty)$ via $\tilde{V}(\bx) := \hat{V}(\mathfrak{n}(\bx))$. (The function $\tilde{V}$ depends only on the quantity of components in the vector $\bx$.) Then there exists a positive constant $C$ such that
$$
\norm{P^t(\bx, \cdot) - \Pi(\cdot)}_{\TV} \le C\tilde{V}(\bx)e^{-\alpha t},\quad \bx \in \mathcal X,\quad t \ge 0.
$$
\end{thm}

\begin{exm}
Assume $\lambda^*_n \le \lambda_* < \kappa_* \le \kappa^*_n$ for  constants $\lambda_*,\, \kappa_*$. Then we can take $\hat{V}(n) = n+1$ and $\alpha := (\kappa_* - \lambda_*)/2$, since the left-hand side of~\eqref{eq:exact-L} is less than or equal to
$$
n\lambda_* - n\kappa_* \le -2n\alpha \le -(n+1)\alpha = -\alpha\hat{V}(n),\quad n = 1, 2, \ldots.
$$
\end{exm}

\section{Large-Scale Behavior: First Setting}\label{sec:mean-field}

To analyze the distribution of bank reserves we consider the following scaling limit as the number of banks tends to infinity.  Fix an index $p \ge 2$. Consider a sequence of systems $(X^{(N)})_{N \ge 1}$ governed by the same dynamics as described in Section 2, with the same parameters $\lambda_{\cdot}(\cdot)$, $\mathcal B_{\cdot, \cdot}$, $\kappa_{\cdot}(\cdot, \cdot)$ such that $\mathfrak n(X^{(N)}(0)) = N$: the system $X^{(N)}$ starts with $N$ banks at time $t = 0$. With the empirical measure $\mu_{\cdot}$ in (\ref{eq: empiricaldist}) let us define the empirical measure process
\begin{equation} \label{eq: muNmutN}
\mu^{(N)} = \bigl(\mu^{(N)}_t,\, t \ge 0\bigr),\quad \mu_t^{(N)} := \mu_{X^{(N)}(t)}.
\end{equation}
We focus on the {\it current level} and {\it current size}
$$
\mathcal N_N(t) := \mathfrak n\left(X^{(N)}(t)\right)\quad\mbox{and}\quad \mathcal S_N(t) = \mathfrak s\bigl(X^{(N)}(t)\bigr),\ t \ge 0,
$$
of the systems,  as well as the {\it current mean reserves}:
$$
\mathbf{m}_N(t) := \frac{\mathcal S_N(t)}{\mathcal N_N(t)} = \overline{X^{(N)}(t)}.
$$

\subsection{McKean-Vlasov jump-diffusions} Now, let us describe the limiting measure-valued process which is a McKean-Vlasov jump-diffusion. This is a generalization of a McKean-Vlasov diffusion (with drift and diffusion coefficients depending not only on the current process, but on its distribution) to a jump-diffusion.

\smallskip

Consider a filtered probability space $(\Omega, \mathfrak F, (\mathfrak F_t)_{t \ge 0}, \mathbb P)$ with the filtration satisfying the usual conditions, and another measurable space $(\mathcal U, \mathfrak U)$ with a finite measure $\mathfrak n$. Fix $p > 1$. Recall that $\mathcal P_p$ is the space of probability measures on $\mathbb R$ with finite $p$th moment, which is a metric space with respect to the Wasserstein distance $\mathcal W_p$.

\smallskip

Assume $W = (W(t),\, t \ge 0)$ is an $(\mathfrak F_t)_{t \ge 0}$-Brownian motion, and $\mathbb{N} = (\mathbb{N}(t),\, t \ge 0)$ is an $(\mathfrak F_t)_{t \ge 0}$-Poisson process with intensity $\lambda$, independent of $W$. Fix drift and diffusion functions $g, \sigma : \mathbb R \times \mathcal P_p \to \mathbb R$, as well as a $\mathcal P_{p}$-valued function $\mu : \mathbb R \times\mathcal P_p \to \mathcal P_p$ for jump size distributions. Also, fix a positive number $\lambda > 0$. A process $Z = (Z(t),\, t \ge 0)$ with paths in the Skorohod space $\cD[0, \infty)$ is called a {\it McKean-Vlasov jump-diffusion} if it satisfies
\begin{equation}
\label{eq:Vlasov}
Z(t) = Z(0) + \int^{t}_{0}\left[ g(Z(s), \nu(s))\,\md s + \sigma(Z(s), \nu(s))\,\md W(s) \right]+ \sum\limits_{k=1}^{\mathbb{N}(t)}\bigtriangleup Z(\tau_k),
\end{equation}
where $0 = \tau_0 < \tau_1 < \tau_2 < \ldots$ are the jump times of the Poisson process $\mathbb{N} = (\mathbb{N}(t),\, t \ge 0)$ with intensity $\lambda$, and $\bigtriangleup Z(t) = Z(t) - Z(t-) \sim \mu_{Z(t-), \nu(t-)}$ for $t \ge 0$. Here, $\nu(t)$ is the distribution of $Z(t)$; and $\nu(t-)$ is the weak limit of $\nu(s)$ as $s \uparrow t$ (similarly to $Z(t-)$). Somewhat abusing the notation, we also call $\boldsymbol{\nu}$, which is the distribution of the process $Z$, a solution to~\eqref{eq:Vlasov}.

\smallskip

To give some explanation about the process $Z$: Between jumps it behaves as a continuous McKean-Vlasov nonlinear diffusion, with drift and diffusion coefficients dependent not only on its current state, but also on its current distribution. The jump measure corresponds to killing  $Z$ with rate $\lambda$, and restarting it according to the measure $\mu_{Z(t-), \nu(t-)}$ at every jump moment $t$.

\smallskip

We now state an existence and uniqueness result for $\mathcal W_{p}$, $p \ge 1$. Its proof is very similar to the result of \cite[Theorem 2.2]{Graham} for $\mathcal W_{1}$.
We refer the interested reader also to \cite{GrahamNew, Funaki}.

\begin{lemma} Fix $p \ge 1$. Assume $g, \sigma$ are jointly Lipschitz (with respect to $\mathcal W_p$ for their second argument), and $h$ is jointly Lipschitz with respect to the $L^p$-norm. That is, there exists a constant $C > 0$ such that for all $x_1, x_2 \in \mathbb R$ and $\zeta_1, \zeta_2 \in \mathcal P_p$, we have:
\begin{align}
\label{eq:Lip-condition}
\begin{split}
|g(x_1, \zeta_1) - g(x_2, \zeta_2)| & \le 
C\left(|x_1 - x_2| + W_p(\zeta_1, \zeta_2)\right),\\
|\sigma(x_1, \zeta_1) - \sigma(x_2, \zeta_2)| & \le 
C\left(|x_1 - x_2| + W_p(\zeta_1, \zeta_2)\right),\\
W_p\left(\mu_{x_1, \zeta_1}, \mu_{x_2, \zeta_2}\right) & \le 
C\left(|x_1 - x_2| + W_p(\zeta_1, \zeta_2)\right).
\end{split}
\end{align}
Take an initial condition $Z(0) \sim \nu(0) \in \mathcal P_p$. Then the equation~\eqref{eq:Vlasov} has a unique solution, which is an element of $\mathcal P_p(\mathcal D[0, T])$ for every $T > 0$.
\label{lemma:Vlasov}
\end{lemma}

\begin{rmk}
Note that within this framework it is possible to accommodate varying intensity of jumps, that is, $\lambda$ dependent on $Z(t)$ and $\nu(t)$. Indeed, assume that $\lambda = \lambda(Z(t), \nu(t))$ is bounded from above by a constant $\overline{\lambda}$. Instead of the measures $\mu_{z, \nu}$, we can consider measures
\begin{equation}
\label{eq:adjust}
\tilde{\mu}_{z, \nu} := \overline{\lambda}^{-1}\left[\lambda(z, \nu)\mu_{z, \nu} + (\overline{\lambda} - \lambda(z, \nu))\delta_z\right]
\end{equation}
under the assumption that the intensity of jumps is now constant and is equal to $\overline{\lambda}$. If $\lambda(z, \nu)$ is Lipschitz in $z$ and $\nu$, and the third among~\eqref{eq:Lip-condition} holds for the family $(\mu_{z,\nu})$, then the family $(\tilde{\mu}_{z, \nu})$ of measures from~\eqref{eq:adjust} also satisfy the third condition in~\eqref{eq:Lip-condition}.
\end{rmk}
Next, we can prove that the McKean-Vlasov jump-diffusion $\, Z\,$ in (\ref{eq:Vlasov}) satisfies for $p \ge 1$
$$
\lim_{s\to t} \E [|Z(s) - Z(t)|^p] \, =\,  0 \qquad \forall t \in [0, T].
$$
This implies that the mapping $t \mapsto \nu(t)$ is continuous, i.e.,~$\nu \in C([0, T], \mathcal W_p)$. We can state the McKean-Vlasov-It\^o process in an equivalent form as a martingale problem. For a function $f \in \mathcal C^2_b$,  a scalar $z \in \mathbb R$, and a probability measure $\nu \in \mathcal P_p$, define
\begin{equation}
\label{eq:generator-McKean}
\cL f(z, \nu) = g(z, \nu)f'(z) + \frac12\sigma^2(z, \nu)f''(z) + \lambda \int_{\mathbb R}\left[f(z + u) - f(z)\right]\mu_{z, \nu}(\md u).
\end{equation}
We say that a probability measure $\boldsymbol{\nu}$ in $\mathcal P_p(\cD[0, T])$ is a solution to a McKean-Vlasov jump-diffusion martingale problem, if for every function $f \in \mathcal C^2_b$ the process
\begin{equation}
\label{eq:mgle-problem}
f(Z(t)) - f(Z(0)) - \int_0^t\cL f(Z(s), \nu(s))\,\md s,\ t \in [0, T],
\end{equation}
is a martingale, where $\nu(s)$ is the projection of $\boldsymbol{\nu}$ at time $s$, and $Z \sim \boldsymbol{\nu}$ is a canonical stochastic process with trajectories in $\cD[0, T]$.  Taking expectations of this martingale, taking derivatives with respect to time, and then using that $Z(t) \sim \nu(t)$, we arrive at the following ODE
\begin{equation}
\label{eq:fund-ODE}
\frac{\md}{\md t}(\nu(t), f) = (\nu(t), \cL f(\cdot, \nu(t))) , \quad t \in [0, T].
\end{equation}
Equation~\eqref{eq:fund-ODE} characterizes the McKean-Vlasov-It\^o equation via martingale problems. The following lemma summarizes our description of McKean-Vlasov jump-diffusions.

\begin{lemma}
Under the assumptions of Lemma~\ref{lemma:Vlasov}, an adapted process $Z = (Z(t),\, t \ge 0)$ with rcll trajectories, with distribution $\nu(t) \sim Z(t)$, is a McKean-Vlasov jump-diffusion as in~\eqref{eq:Vlasov}, if and only if for every test function $f \in \mathcal C^2_b$, the process in~\eqref{eq:mgle-problem} is a martingale; or,  equivalently, if for every test function $f \in \mathcal C^2_b$, the equation~\eqref{eq:fund-ODE} holds.
\label{lemma:equiv-technical}
\end{lemma}

The proof of Lemma~\ref{lemma:equiv-technical} follows standard arguments (see for example \cite[Section 5.4]{KSBook} or \cite[Section 4.4]{EthierKurtz}) and is therefore omitted.

\begin{rmk} \label{rmk:history}
In Section~\ref{sec:mf-N} we shall need a version of \eqref{eq:Vlasov} with parameters $g,\, \sigma,\, \mu$, depending not only on $Z(t)$ and $\nu(t)$, but on the whole history $Z(s)\,, \nu(s)\,, 0 \le s \le t$, as well as on time $t$. Thus, this McKean-Vlasov jump-diffusion is {\it path-dependent} and {\it time-inhomogeneous}. Similar to~\eqref{eq:Lip-condition} we then modify the Lipschitz conditions as follows: For $t \in [0, T]$, $z_1, z_2 \in D[0, t]$, and $\boldsymbol{\nu}_1, \boldsymbol{\nu}_2 \in C([0, t], \mathcal P_p)$, with $\nu_1(s)$, $\nu_2(s)$ being push-forwards of $\boldsymbol{\nu}_1$, $\boldsymbol{\nu}_2$ with respect to the projection mapping $x \mapsto x(s)$ for each $0 \le s \le t$, we define the distance function:
$$
\Delta(t) := \sup\limits_{0 \le s \le t}|z_1(s) - z_2(s)| + \sup\limits_{0 \le s \le t}\mathcal W_p(\nu_1(s), \nu_2(s)),
$$
and impose the following Lipschitz conditions: 
\begin{align}
\label{eq:new-lip}
\begin{split}
\left|g(t, z_1, \boldsymbol{ \nu}_1) - g(t, z_2, \boldsymbol{ \nu}_2)\right| & 
\le C\cdot\Delta(t),\\
\left|\sigma(t, z_1, \boldsymbol{\nu}_1) - \sigma(t, z_2, \boldsymbol{\nu}_2)\right| & \le C\cdot\Delta(t),\\
\mathcal W_p\left(\mu_{t, z_1, \boldsymbol{\nu}_1}, \mu_{t, z_2, \boldsymbol{\nu}_2}\right) & \le C\cdot\Delta(t).
\end{split}
\end{align}
Then Lemma~\ref{lemma:Vlasov} and Lemma~\ref{lemma:equiv-technical} still hold, with the formula for the generator~\eqref{eq:generator-McKean}.
\end{rmk}

\subsection{Mean-field limit: first main result} We investigate the limiting behavior of the empirical measure process $(\mu^{(N)}_t,\, t \ge 0)$ as $N \to \infty$. We shall show that these measure-valued processes converge, in fact, to a deterministic measure-valued process, governed by a certain McKean-Vlasov equation. To this end, we impose some additional assumptions on the parameters of our model as the number of banks $n$ tends to infinity. Note that in our scaling, we re-parametrize in terms of $n$ and $y= s/n$ (i.e.,~$s = n y$) in reference to the mean size $\mathbf{m}$ above.

\begin{asmp} As $n \to \infty$, $\mathcal B_{n, ny} \to \mathcal B_{\infty, y}$ in the Wasserstein distance $\mathcal W_p$, uniformly over $y > 0$, with the family $(\mathcal B_{\infty, y})_{y > 0}$ continuous in $\mathcal W_p$; and the measures $\mathcal B_{n, s},\, n \ge 0,\, s> 0$ have uniformly bounded $p$-th moments.
\label{asmp:birth-measures}
\end{asmp}

\begin{asmp}  As $n \to \infty$, we assume uniform convergence to a continuous limit $\lambda_{\infty}$:
$$
\frac{\lambda_n(ny)}{n} \to \lambda_{\infty}(y),\, y > 0
$$
uniformly in $y >0$. Moreover, there exists a constant $C_{\lambda}$ such that $\lambda_n(s) \le C_{\lambda}(n + s)$ for all $n,s$.
\label{asmp:birth-intensities}
\end{asmp}

Examples of birth rates satisfying Assumption ~\ref{asmp:birth-intensities} are $\lambda_{n}(s) = \bar{\lambda}s$ (new banks formed at rate proportional to total reserves)  and $\lambda_n(s) = \bar{\lambda}n$ (new banks formed at rate proportional to current number) for a constant $\bar{\lambda} > 0$,  which both lead to $\lambda_{\infty}(y) = \bar{\lambda}y$. Note that birth rates must increase as system size $N$ grows to avoid the trivial limit $\lambda_\infty = 0$. 

\begin{asmp}
\label{asmp:default-measures}
If $\xi_{n, s, x} \sim \cD_{n, s,x}$, then $n\xi_{n, ny, x} \to \xi_{\infty, y,x} \sim \cD_{\infty, y, x}$ as $n \to \infty$ in the Wasserstein distance $\mathcal W_p$ uniformly over all $x, y > 0$, where the family of measures $(\cD_{\infty, y,x})_{x, y > 0}$ is continuous in $\mathcal W_p$ jointly in $x$ and $y$; support of measures $\cD_{\infty, y, x}$ is bounded from above uniformly in $x$ and $y$; and $\xi_{\infty, y, x}$ has uniformly bounded $p$th moment over all $x, y$. 
\end{asmp}

\begin{rmk}
From~\eqref{asmp:default-measures} it follows that for $q \le p$,
\begin{equation}
\label{eq:D-const}
C_{\cD, q} := \sup\limits_{N, x, s}\bigl[N^q\int_0^{1}z^q\,\cD_{N, x, s}(\mathrm{d}z)\bigr] < \infty,
\end{equation}
and there exist $n_0$ and $\varepsilon_0 \in (0, 1)$ such that a.s. for all $n \ge n_0$, $0 < x < ny$, $|\xi_{n, y, x}| \le 1 - \varepsilon_0$.
\label{rmk:technical-default-measures}
\end{rmk}

The requirement in Assumption~\ref{asmp:default-measures} is that the default impact decreases inversely proportional to the scaling parameter. Larger banking systems will experience more defaults (namely proportionally to $N$, see the next assumption), so the impact of each default must shrink to compensate. Note that the limiting distribution $\cD_{\infty,y,x}$ does not matter and only its mean will appear in the limit equation. An example would be $\xi_{n,s,x} \sim \Uni( 0, \bar{d}/n)$ so that $\xi_{\infty, y,x} \sim \Uni(0, \bar{d})$. The next assumption is about the convergence of the default rates. Another example would be the case when default rates are independent of $n,s$: $\kappa_\infty(x) = \kappa_n(x)$.

\begin{asmp}
\label{asmp:default-intensities}
As $n \to \infty$, uniformly over $y, x \in (0, \infty)$, we have: $\kappa_n(ny,x) \to \kappa_{\infty}(y,x)$, with $\kappa_{\infty}(y,x)$ continuous in $y$. Moreover, there exists a constant $C_{\kappa}$ independent of $x$, $n$ and $s$ such that $\kappa_n(s,x) \le C_{\kappa}$ for all $n,s,x$.
\end{asmp}

Denote the means of the limiting measures $\mathcal B_{\infty, y}$ and $\cD_{\infty, y, x}$ by $\overline{\mathcal B}_{\infty}(y)$ and $\overline{\cD}_{\infty}(y,x)$. Define
\begin{equation}
\label{eq:expression-for-psi}
\psi(x,y) := r - \overline{\cD}_{\infty}(y, x)\kappa_{\infty}(y,x),\qquad y > 0,
\end{equation}
\begin{equation}
\label{eq:G-adjusted}
\cG_y f(x) := \psi(x, y)D_1f(x) + \frac{\sigma^2}2D_2f(x) = \cG f(x) - \overline{\cD}_{\infty}(x, y)\kappa_{\infty}(x, y)D_1f,
\end{equation}
where $\cG$ is from \eqref{eq:operators}. This will be the limiting diffusion term, corresponding to the original geometric Brownian motion dynamics, summarized by $\cG$ plus the additional mean-field-based drift term due to the default interactions. Define the measure-valued process $\mu^{(\infty)} = (\mu^{(\infty)}_t,\, t \ge 0)$ as the law of a McKean-Vlasov jump-diffusion with generator
\begin{align}
\label{eq:main-generator}
\begin{split}
\mathcal L_{\nu}f(z) & := \cG_{\overline{\nu}}f + \lambda_{\infty}\left(\overline{\nu}\right)\int_0^{\infty}\left[f(w) - f(z)\right]\,\mathcal B_{\infty, \overline{\nu}}(\mathrm{d}w) \\ & +  \kappa_{\infty}(\overline{\nu}, z)\int_0^{\infty}\left[f(w) - f(z)\right]\,\nu(\mathrm{d}w).
\end{split}
\end{align}
We can apply current distribution $\nu$ to this generator and get:
\begin{align}
\label{eq:main-operator}
\begin{split}
\mathcal A(\nu, f) := (\nu, \mathcal L_{\nu}f) & =  (\nu, \cG_{\overline{\nu}}f)  + \lambda_{\infty}\left(\overline{\nu}\right) \big[ \left(\mathcal B_{\infty, \overline{\nu}}, f\right)  - \left(\nu, f\right) \big] \\  &   + (\nu, f) \left(\nu, \kappa_{\infty}\left(\overline{\nu}, \cdot\right) \right) -  \left(\nu, \kappa_{\infty}\left(\overline{\nu}, \cdot\right)f\right).
\end{split}
\end{align}
The main result below is that $\mu^{(\infty)}$ is a suitable limit of $\mu^{(N)}$'s from (\ref{eq: muNmutN}). To explain the form of $\mathcal{A}$ we discuss each term. First, the $\cG_{\overline{\nu}}$ term arises from the additional average downward drift from the defaults. Next, there are two different jump mechanisms:  The second piece
\begin{equation}
\label{eq:birth-term}
\lambda_{\infty}\left(\overline{\nu}\right)\left(\mathcal B_{\infty, \overline{\nu}}, f\right) - \lambda_{\infty}\left(\overline{\nu}\right)\left(\nu, f\right)
\end{equation}
arises from births from the pre-limit finite system which translate into killing and restarting according to the measure $\mathcal B_{\infty, \cdot}$. This can be viewed as exogenous ``regeneration'' with a source measure $\mathcal B_{\infty,\overline{\nu}}$.
 The third piece
\begin{equation}
\label{eq:default}
(\nu, f) \left(\nu, \kappa_{\infty}\left(\overline{\nu}, \cdot\right)\right) - \left(\nu, \kappa_{\infty}\left(\overline{\nu}, \cdot\right)f\right)
\end{equation}
is an endogenous push due to the non-constant default intensity. Regions where $\kappa_\infty$ is higher experience higher rates of defaults, whereby the respective banks ``dis-appear''; in the limit they immediately ``re-appear'' according to $\nu$. This can be thought of as a genetic mutation: particles in high-default regions get killed and replaced with new particles sampled according to $\nu$.

\smallskip

If $\kappa_{\infty}(y, x) = \kappa_{\infty}(y)$ depends only on $y$, then the term~\eqref{eq:default} vanishes. Indeed, this term then becomes proportional to the action $(\nu, f)$ of the current distribution $\nu$ on the test function $f$. This means we kill the process and restart it at the same distribution, which is equivalent to doing nothing. Thus, only the decrease of reserves of  all remaining banks by i.i.d.~fractions influences the empirical measure, turning the drift from $r$ into $\psi$ from~\eqref{eq:expression-for-psi}.

\smallskip

Financially, we see that defaults from the pre-limit finite system translate into two effects. On the one hand, defaults themselves create additional downward drift inside $\psi$ from~\eqref{eq:expression-for-psi}, as compared with the original drift $r$. On the other hand, financial contagion after a default creates reset times when the process is killed and restarted, which corresponds to the term~\eqref{eq:default}. Let us mention how bankruptcies occur in this limit: The fraction of banks defaulting at time $[t, t + \Delta t]$  is $\kappa_{\infty}(\mathbf{m}(t))\,\mathrm{d}t$. That is, the fraction of banks defaulted during time interval $[s, t]$ is $\int_s^t\kappa_{\infty}(\mathbf{m}(u))\,\mathrm{d}u$.  This fraction can be greater than $1$, because new banks emerge all the time. 

\smallskip
In the notation of \eqref{eq:Vlasov}, we interpret the mean field limit as a McKean-Vlasov jump-diffusion $Z$ which has drift $\psi(Z(t), \nu_t)Z(t)$, diffusion $\sigma Z(t)$ and the following family of jump measures:
\begin{equation}
\label{eq:jump-measures}
\mu_{z, \nu}(\mathrm{d}w) = \lambda_{\infty}\left(\overline{\nu}\right)\,\mathcal B_{\infty, \overline{\nu}}(\mathrm{d}w) +  \kappa_{\infty}(\overline{\nu}, z)\,\nu(\mathrm{d}w).
\end{equation}
The process $Z$ can be thought of as a `representative particle'. 

\begin{lemma}
Under Assumptions~~\ref{asmp:birth-measures},~\ref{asmp:birth-intensities},~\ref{asmp:default-measures},~\ref{asmp:default-intensities}, there exists a unique McKean-Vlasov jump-diffusion $Z$ with generator~\eqref{eq:main-operator}.
\end{lemma}

\begin{proof} We verify the conditions of Lemma~\ref{lemma:Vlasov}. Both terms in~\eqref{eq:jump-measures}:
$$
\lambda_{\infty}\left(\overline{\nu}\right)\mathcal B_{\infty, \overline{\nu}}(\mathrm{d}w)\quad \mbox{and}\ \quad  \kappa_{\infty}(z, \overline{\nu})\,\nu(\mathrm{d}w)
$$
have the $p$th moment uniformly bounded for $p \ge 2$. Note the Lipschitz property of these measures with respect to the measure $\nu$ in $\mathcal W_p$. Together with the Lipschitz property of the functions $x \mapsto \psi(x, y)x$ and $x \mapsto \sigma x$, this completes the proof.
\end{proof}

The following is the main result of this section, with proof postponed to Section \ref{sec:theorem41}.

\begin{thm} Suppose the initial empirical measures converge in $\mathcal W_p$ with $p > 1$:
$$
\mu^{(N)}_0 \Rightarrow \mu^{(\infty)}_0\ \mbox{as}\ N \to \infty.
$$
Under Assumptions~\ref{asmp:birth-measures},~\ref{asmp:birth-intensities},~\ref{asmp:default-measures},~\ref{asmp:default-intensities}, we have the following convergence in law in the Skorohod space $\cD([0, T], \mathcal P_{q})$, for every $T > 0$ and $q \in (1, p)$:
$$
\mu^{(N)} \Rightarrow \mu^{(\infty)}\ \mbox{as}\ N \to \infty.
$$
\label{thm:hydro}
\end{thm}

By Lemma \ref{lemma:technical-convergence}, the functional $\nu \mapsto \overline{\nu}$ (taking the mean) is continuous in $\mathcal P_q$. Thus we have:

\begin{cor} As $N \to \infty$, we have weak convergence in $\cD[0, T]$:
$$
\mathbf{m}_N(\cdot) \Rightarrow \mathbf{m}(\cdot) \quad \mbox{where}\quad \mathbf{m}(t) := \E[ Z(t)] = (\mu^{(\infty)}_{t},f_{1}),
$$
\end{cor}

\subsection{Default intensity independent of size} Under Assumptions~\ref{asmp:birth-measures},~\ref{asmp:birth-intensities},~\ref{asmp:default-measures},~\ref{asmp:default-intensities}, if the killing rate and the mean of the contagion measure
$$
\kappa_{\infty}(y, x) = \kappa_{\infty}(y)\quad \mbox{and}\quad \overline{\mathcal D}_{\infty}(y, x) \, =\,  \overline{\mathcal D}_{\infty}(y)
$$
are independent of the individual size $x$ but dependent only on the average of the system, we can solve the McKean-Vlasov equation explicitly. Indeed, in this case, in the limit $N \to \infty$ default intensities and default impacts are independent of the size of defaulting banks. In this case, we can rewrite~\eqref{eq:expression-for-psi} as $\psi(x, y) = \psi(y) = r -  \overline{\cD}_{\infty}(y)\kappa_{\infty}(y)$. Then we rewrite the McKean-Vlasov equation for $\mu^{(\infty)}$ as follows:
\begin{equation}
\label{eq:McKean}
Z(t) = Z(0) + \int_0^t\left[\psi(\mathbf{m}(s))\,Z(s)\,\md s  + \sigma Z(s)\,\md B(s)\right] + \sum\limits_{k=1}^{\mathbb N(t)}\Delta Z(\tau_k),\quad\, \mathbf{m}(t) = \E [Z(t)],
\end{equation}
with $B = (B(t),\, t \ge 0)$ being a Brownian motion; $\mathbb N(t)$ is a time-nonhomogeneous Poisson process with rate $\lambda_{\infty}(\mathbf{m}(t))$, with jump times $\tau_k$, and $Z(\tau_k)  \sim \mathcal B_{\infty, \mathbf{m}(\tau_k-)}$.  Assuming the function $\psi$ is Lipschitz continuous, equation \eqref{eq:McKean} has a unique solution for any initial condition, see for example \cite{Funaki}. Let us now solve \eqref{eq:McKean}. Its parameters: drift, volatility, and jump measure $\lambda_{\infty}(\cdot)\mathcal B_{\infty, \cdot}$, depend on the distribution of $Z(t)$ only through its mean $\mathbf{m}(t)$. Therefore, we can solve first for $\mathbf{m}(\cdot)$ and then for $Z(t)$. Take expectations in \eqref{eq:McKean}: 
\begin{equation}
\label{eq:ODE}
\mathbf{m}'(t) = \psi(\mathbf{m}(t))\mathbf{m}(t) + \lambda_{\infty}(\mathbf{m}(t))\left(\overline{\mathcal B}_\infty(\mathbf{m}(t)) - \mathbf{m}(t)\right), \quad \mathbf{m}(0) = \int_0^{\infty}x\mu^{(\infty)}_0(\md x).
\end{equation}
Assuming this (deterministic) ODE has a unique solution $\mathbf{m}(\cdot)$ we plug it in~\eqref{eq:McKean} to obtain that $Z$ is a geometric Brownian motion with time-dependent drift:
\begin{equation}
\label{eq:GBM}
Z(t) = Z(0)\exp\left[\int_0^t\left[\psi(\mathbf{m}(s)) - \sigma^2/2\right]\,\md s + \sigma B(t)\right],
\end{equation}
killed at rate $\lambda_{\infty}(\mathbf{m}(t))$, and resurrected according to $\mathcal B_{\infty, \mathbf{m}(t)}$. Let us find constant solutions $\mu^{(\infty)}_t \equiv \Pi$, or, equivalently, stationary solutions for the process $Z$ in \eqref{eq:GBM}. For any such solution, its mean $\mathbf{m}(t) \equiv M$ is also independent of $t$. Therefore,  we let the right-hand side of the ODE~\eqref{eq:ODE} to be equal to zero. This is an algebraic equation:
\begin{equation}
\label{eq:stat}
\psi(M)M + \lambda_{\infty}(M)\left(\overline{\mathcal B}_{\infty}(M) - M\right) = 0.
\end{equation}
For every solution $M > 0$ of this equation (which is notably independent of $\sigma$), from~\eqref{eq:GBM} we get geometric Brownian motion:
$$
\md Z(t) = Z(t)\left[\psi(M)\,\md t + \sigma\,\md B(t)\right],
$$
killed at constant rate $\lambda_{\infty}(M)$, and resurrected according to the probability measure $\mathcal B_{\infty, M}$. The most elementary case is when all limiting parameters are constant:
\begin{equation}
\label{eq:const-param}
\lambda_{\infty},\, \kappa_{\infty},\, \overline{\mathcal B}_{\infty},\,\overline{\cD}_{\infty}.
\end{equation}
Then the differential equation~\eqref{eq:ODE} takes the form
$$
\mathbf{m}'(t) = \lambda_{\infty}\overline{\mathcal B}_{\infty} + \gamma \mathbf{m}(t), \quad \text{ where } \quad \gamma := \left(r - \overline{\cD}_{\infty}\kappa_{\infty} - \lambda_{\infty}\right).
$$
Given the initial condition $\mathbf{m}(0)$, the solution of this first-order linear equation is
\begin{align}\label{eq:m-ode}
\mathbf{m}(t) = \left(\mathbf{m}(0) - \frac{\lambda_{\infty}\overline{\mathcal B}_{\infty}}{\gamma} \right)e^{-\gamma t} + \frac{\lambda_{\infty}\overline{\mathcal B}_{\infty}}{\gamma}.
\end{align}
If $\gamma \ne 0$, there exists a unique solution to the algebraic equation~\eqref{eq:stat}, which is the limit for the solution $ \mathbf m(t)$ of the differential equation~\eqref{eq:m-ode}:
$M = \gamma^{-1}\lambda_{\infty}\overline{\mathcal B}_{\infty} = \lim_{t \to \infty} \mathbf{m}(t)$.
The left panel of Figure \ref{fig:conv} illustrates the mean field limit in the constant default intensity case. We take $\lambda_n(s) = 0.2n, \kappa_n(s,x) = 0.1$, $\mathcal{B}_{n,s} \sim \Exp(1)$ and $\mathcal{D}_{n,s,x} \sim \Uni(0, n^{-1})$, and $r=0.05$. Note that in the mean field limit $\lambda_\infty = 0.2$, $\kappa_\infty(x) = 1$ and $\overline{\mathcal{B}}_\infty = 1, \overline{\mathcal{D}}_\infty = 0.5$, leading to 
$$
M = \frac{0.2\cdot 1}{0.2 + 0.1 \cdot 0.5 - 0.05} = 1.
$$
In Figure~\ref{fig:conv} we initialize with $\mu_0^{(N)} \sim \Exp(0.5)$ so that $\mathbf{m}(0) = 2$ and the solution of \eqref{eq:m-ode} reads as $\mathbf{m}(t) = 1 \exp( - 0.2t) + 1$. The figure shows the simulated distribution of $\mathbf{m}(t)$ based on running 100 paths of the pre-limit system $X^{(N)}$ with $N=5,25, 100$. For each run $i=1,\ldots,100$ we compute the resulting $m^{i}_N(t)$ as the empirical average bank size at step $t$ and finally plot $\Ave( m^{i}_N(t) )$, as well as the $5\%-95\%$ quantiles of $m^{i}_N(t)$ across the $100$ runs. The latter visualize the variance of $m_N(t)$; as expected as $N$ increases, $m_N(t)$ converges in distribution to the deterministic limit $\mathbf{m}(t)$ reported above. We note that in this example due to the limited interaction among the banks and the light-tailed  default and birth distributions, the convergence is very rapid so already $\E[ m_N(t) ] \simeq \mathbf{m}(t)$ even for very small $N=5$.

\begin{figure}[t]
\begin{center}
\includegraphics[height=3in,width=3in]{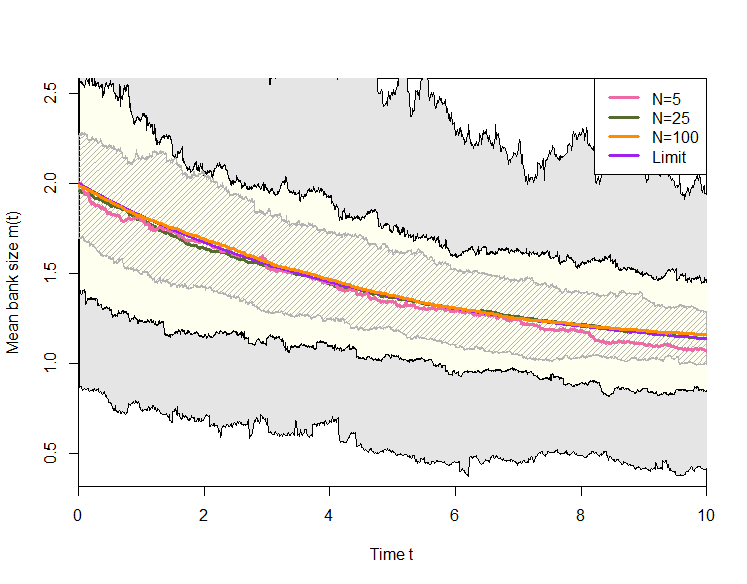}
\includegraphics[height=3in,width=3in]{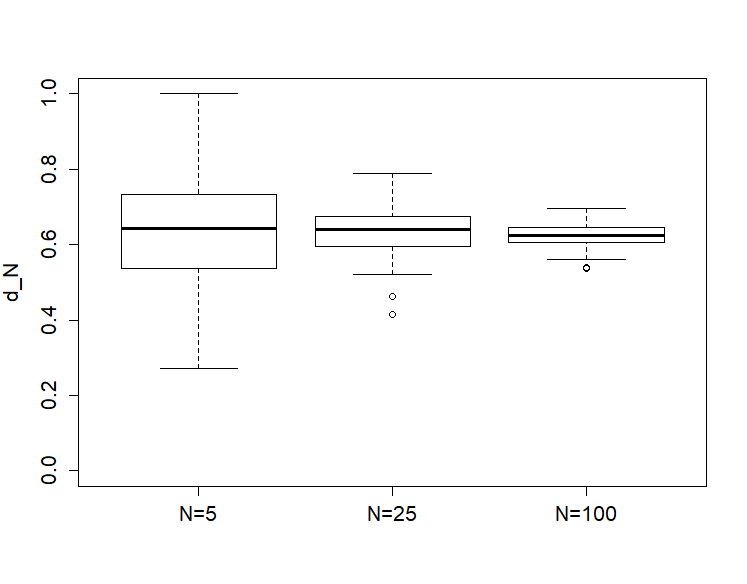}
\end{center}
\caption{Left panel: distribution of $m_N(t)$ on $t \in [0,10]$ and $N=5,25,100$ based on $100$ simulated trajectories of $X^{(N)}$. The initial distribution is $\mu^{(N)}_0 \sim \Exp(0.5)$ so that $m_N(0) = 2$.  Right panel: distribution of $d_N := (\mu^{(N)}_T, 1_{[0,D]})$, proportion of banks with reserves less than $D=1$ at $T=10$.}
\label{fig:conv}
\end{figure}


\subsection{Capital distribution}
The mean field limit offers insight into the bank reserves distribution which is key to analyzing the probability of {\it systemic events}: when many banks default or have low reserves. For example, in structural models there is typically a risk threshold $D > 0$ so that banks whose reserves are below $D$ are viewed as insufficiently capitalized. Taking $f_D(x) := 1_{\{x \le D\}}$, the systemic risk of the banking network at epoch $t$ can be assessed as
\begin{align}\label{eq:systemic-risk}
\frac{|\#\{i \in I(t)\mid X_i(t) \le D\}|}{N(t)} = \bigl(\mu_t^{(N)}, f_D\bigr).
\end{align}
As $N \to \infty$, empirical measures $\mu^{(N)}_t$ converge in $\mathcal P_q$ (and therefore weakly) to a deterministic measure
$\mu^{(\infty)}_t$, which is absolutely continuous and hence $(\mu^{(N)}, f_D) \to \mu^{(\infty)}_t(0,D)$. At the same time, as $t$ becomes large, $\mu^{(\infty)}_t$ converges to its stationary distribution $\Pi$, so the fraction of banks below $D$ approaches $\Pi(0,D]$.

The right panel of Figure \ref{fig:conv} shows the distribution of $d_N := (\mu^{(N)}_t, f_D)$ at fixed $t$ as we vary $N$. Specifically, we use the same setting as in the left panel of that Figure and take $D = 1$. As expected, $(\mu^{(N)}_t, f_D)$ becomes more deterministic as $N$ grows and the empirical fluctuations decrease. In the Figure, we see that about 60\% of the banks will have assets below $D=1$ at $T=10$. The take-home message is that analysis of $\Pi$ (and $\mu^{(\infty)}_\cdot$ for shorter-term objectives) holds the key for understanding the financial riskiness of the system, for example whether the banks tend to cluster into distinct groups (small banks, large banks, etc.)



\subsection{Propagation of chaos} Let us further describe the behavior of a typical bank  as the number of banks tends to infinity. Consider, for example, the first bank $X_1$ starting from time $t = 0$. 

\begin{thm} Assume $X_1^{(N)}(0)$ is deterministic for every $N \ge 1$, and $X_1^{(N)}(0) \to x_1$ as $N \to \infty$. As $N \to \infty$, $X_1^{(N)} \Rightarrow X_1^{(\infty)}$ weakly in $\cD[0, T]$, where $X_1^{(\infty)}$ is a solution to the following  stochastic differential equation:
\begin{align}\label{eq:x-infty-sde}
\md X_1^{(\infty)}(t) = \psi\bigl(X_1^{(\infty)}(t), \mathbf{m}(t)\bigr)\,X_1^{(\infty)}(t)\,\md t + \sigma X_1^{(\infty)}(t)\,\md W(t),
\end{align}
starting from $x_1$, killed with rate $\kappa_{\infty}(\mathbf{m}(t),X_1^{(\infty)}(t))$.
\label{thm:individual}
\end{thm}
The proof of Theorem~\ref{thm:individual} is in Section~\ref{sec:theorem42}.
Observe that compared to \eqref{eq:Vlasov}, the limiting dynamics of $X_1$ are simpler: there is still a mean-field interaction through $\mathbf{m}(t)$, but solely via a mean-field killing rate. Births and hence jumps disappear. We can state this result as follows, recalling the definition of the generator~\eqref{eq:G-adjusted}: \eqref{eq:x-infty-sde} is a McKean-Vlasov diffusion with generator 
\begin{align}
\label{eq:A-star}
  \mathcal A^*_{\nu}f(x_1) = \cG_{\overline{\nu}}f(x_1) - \kappa_{\infty}\left( \overline{\nu},x_1\right)f(x_1).
\end{align}
Similarly to Theorem~\ref{thm:individual}, we have {\it propagation of chaos.} Namely, consider the first $k$ banks instead of only the first one: 
$
(X_1^{(N)}, \ldots, X_k^{(N)}).
$
One can show that the resulting limit  in $\cD([0, T], \mathbb R^k)$ as $N \to \infty$ is a vector of $k$ independent copies of the killed geometric Brownian motion described above: Dependence between the banks vanishes in the limit.

\smallskip

Financially, propagation of chaos offers two convenient features: (1) it abstracts away the complex bilateral dependencies that may exist between individual banks; (2) it distinguishes clearly between the global recurrent nature of the banking system and the individual banks that have finite lifetime (assuming suitable conditions on $\kappa_\infty$ which are expected to hold in realistic settings). The latter is the major difference between a representative particle $Z$ that is infinite-lived, and the prototypical bank $X_1$ that lives for some time and eventually defaults.

\begin{figure}[t]
\begin{center}
\includegraphics[scale=0.5, trim=0.25in 0.25in 0.25in 0.25in]{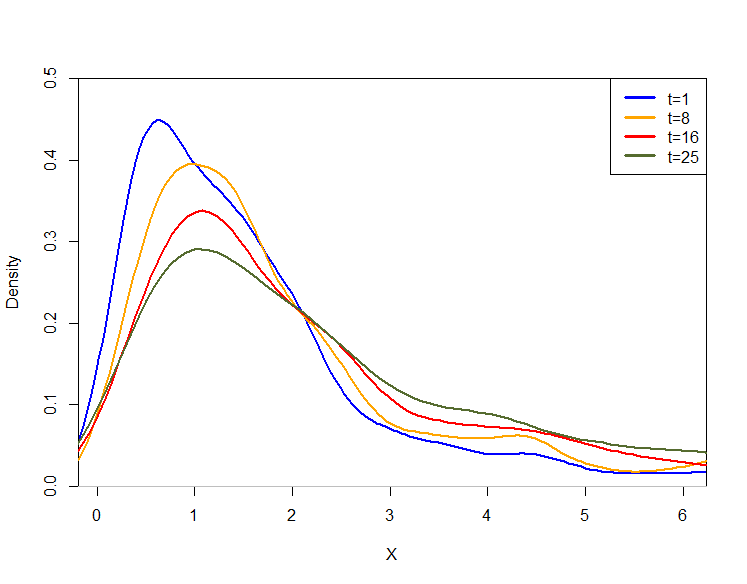}
\end{center}
\caption{Density of $\nu_t$ approximated through an empirical distribution of $\check{X}$ of size $N'=500$ (smoothed via built-in kernel density estimator in MATLAB) at $t=1,8,16,25$, with parameters inherited from Figure 1. The initial condition $\nu_0$ has the law of $e^Z$ for $Z \sim \mathcal N(0, 1)$ with time step $\Delta t=10^{-3}$.} 
\label{fig:mkv}
\end{figure}

\subsection{Illustrating the McKean-Vlasov equation} The limiting McKean-Vlasov equation can be studied using Monte Carlo approximation. Namely, the measures $\mu^{(\infty)}_t$ can be approximated through an empirical distribution of a system $\check{X}$ of $N'$ interacting particles. The particles follow the dynamics of the dummy $\{Z(t)\}$, i.e.,~behave like geometric Brownian motions that are killed and restarted. Note that in contrast to the pre-limit systems $X^{(N)}$, $\check{X}$ has a fixed dimension, $\mathfrak{n}( \check{X}(t) ) = N' \, $ for every $t$. Thus, its dynamics are only in terms of the empirical mean $\check{m}(t) := \frac{1}{N'} \sum_{i=1}^{N'} \check{X}^{(i)}(t)$, rather than system size $N$ and sum $S$. In turn, we may simulate $\check{X}$ using standard tools, for example an Euler scheme with a fixed time-step $\Delta t$.

\smallskip

To do so, each particle $\check{X}^{(i)}$ follows on $[t,t+\Delta t]$ the geometric Brownian motion dynamics with drift $\psi( \check{X}^{(i)}(t), \check{m}(t))$ and volatility $\sigma$, driven by independent Brownian motions $\check{W}_i(\cdot)$. In addition, each particle carries two exponential clocks that fire off at rates $\lambda_\infty( \check{m}(t) )$ and $\kappa_\infty( \check{m}(t), \check{X}^{(i)}(t))$ respectively. Alarms of the first type result in regeneration, i.e.,~the respective particle instantaneously jumps from its current location $\check{X}^{(i)}(t-)$ to a location $\zeta \sim \mathcal{B}_{\infty, \check{m}(t)}$, generated independently of everything else. Alarms of the second type result in resampling due to non-uniform default rates: the particle jumps from $\check{X}^{(i)}(t-)$ to the location of another particle $j$, $\check{X}^{(i)}(t) = \check{X}^{(j)}(t)$, with index $j$ sampled uniformly from $\{1, 2, \ldots, i-1, i+1, \ldots, N'\}$. After this mutation procedure, which can be interpreted as killing particle $i$ and replacing it with a child of particle $j$, the two ``sibling'' particles resume independent movements as geometric Brownian motions. 

\smallskip

Figure \ref{fig:mkv} shows the distribution of the McKean-Vlasov solution: the density $\nu_t$ (which has no closed-form expression) for several values of $t$ with the state-dependent default rates. We take limits of parameters from Figure 1. That is, 
$$
\lambda_{\infty}(y) = 1,\quad \mathcal B_{\infty, y} \sim \Exp(1),\quad \kappa_{\infty}(y, x) = \frac{0.2}{0.01+x},\quad \mathcal D_{\infty, y, x} \sim \Uni[0, 1].
$$

\section{Large-Scale Behavior: Second Setting}\label{sec:mf-N}

For a systemic risk application, our main interest is to build a model with a stationary $\{ N(t) \}$. Indeed, we wish to have a dynamic banking network that expands and shrinks over time but is globally infinite-lived, even if individual banks have finite lifetimes. However, observe that in the setup above, asymptotically both the birth rate $\lambda_n(s)$ and the aggregate default rate $n \cdot \kappa_n(s,x)$ are linear in $n$. Thus  they are comparable, and either births or defaults will ultimately dominate, so that the number of banks $N(t)$ will exponentially grow/shrink in $t$. In other words, starting with a finite $N(0)$, $\E[N(t)]$ will then either exponentially grow to $+\infty$ or exponentially collapse to $0$, neither of which are financially plausible.

\smallskip

To circumvent this issue (which is ultimately not important in the mean-field limit), in this section we consider the case when all parameters of the $N$th system, with $N$ the initial number of banks, are independent of $n$, the current number of banks, but depend on the initial $N$. The motivation is to have models with \emph{constant} birth rates, whereby $\{N(t)\}$ roughly behaves as a linear birth-and-death process with the classical Poisson stationary distribution. We then again scale the systems to recover a (different!) McKean-Vlasov limit. This setting also intrinsically ensures the global recurrence.

\smallskip

To do so, we need to adjust Assumptions~\ref{asmp:birth-measures},~\ref{asmp:birth-intensities},~\ref{asmp:default-measures},~\ref{asmp:default-intensities}, accordingly. Everywhere instead of subscript $n$ we now write $N$, because we now index parameters by the initial size $N$.
Consider a sequence $(X^{(N)})_{N \ge 1}$ of banking systems with the initial values
$X^{(N)}(0) = x^{(N)}_0 \in \mathbb R^{N},\quad \mathfrak n(x^{(N)}_0) = N$. 
The $N$th system $X^{(N)}$ is governed by birth intensities $\lambda_{N, n}(ny) = \lambda_N(y)$, birth measures $\mathcal B_{N, n, ny} = \mathcal B_{N, y}$, default intensities $\kappa_{N, n}(ny, x) = \kappa_N(y, x)$, and default contagion measures $\mathcal D_{N, n, ny, x} = \mathcal D_{N, y, x}$. In all these assumptions, we abuse the notation by dropping the dependence on $n$ (there is now only indirect dependence through $y = s/n$, the mean of the system).

\begin{asmp} As $N \to \infty$, $\mathcal B_{N, y} \to \mathcal B_{\infty, y}$ in the Wasserstein distance $\mathcal W_p$, uniformly over $y > 0$, with the family $(\mathcal B_{\infty, y})_{y > 0}$ continuous in $\mathcal W_p$; and the measures $\mathcal B_{N, y}$ have uniformly bounded $p$-th moments for all $N, y$.
\label{asmp:birth-measures-new}
\end{asmp}

\begin{asmp}  As $N \to \infty$, we assume uniform convergence to a continuous limit $\lambda_{\infty}$:
$$
\frac{\lambda_N(y)}{N} \to \lambda_{\infty}(y),\, y > 0
$$
uniformly in $y >0$; and for some constant $C_{\lambda}$, we have
$\lambda_N(y) \le C_{\lambda}(N + y)$ for all $N,y$.
\label{asmp:birth-intensities-new}
\end{asmp}

\begin{asmp}
\label{asmp:default-measures-new}
If $\xi_{N, y, x} \sim \cD_{N, y, x}$, then $N\xi_{N, y, x} \to \xi_{\infty, y,x}$ as $N \to \infty$ in the Wasserstein distance $\mathcal W_p$ uniformly over all $x, y > 0$, where the family of measures $(\xi_{\infty, y,x})_{x, y > 0}$ is continuous in $\mathcal W_p$ jointly in $x$ and $y$; and $N\xi_{N, y, x}$ has uniformly bounded $p$th moment over all $N, x, y$. We denote the corresponding limiting measure by $\cD_{\infty, y,x}$. \end{asmp}

\begin{asmp}
\label{asmp:default-intensities-new}
As $N \to \infty$, uniformly over $y, x \in (0, \infty)$, we have: $\kappa_N(y,x) \to \kappa_{\infty}(y,x)$, with $\kappa_{\infty}(y,x)$ continuous in $y$. Moreover, there exists a constant $C_{\kappa}$ independent of $x, y, N$ such that $\kappa_N(y,x) \le C_{\kappa}$ for all $N ,x, y$.
\end{asmp}

\begin{exm}
Continuing the example from Figure~\ref{fig:conv}, we take $\lambda_N(s) = \bar{\lambda} N = 0.2 N$; $\kappa_N(s,x) =  \kappa_\infty(x) = 0.1$; $\cD_{N,s,x} = \Uni(0, 1/N)$, so that $\cD_{\infty,s,x} = \Uni(0, 1)$, and $\mathcal{B}_{N,n,s} = \Exp(1)$ so that $\mathcal{B}_\infty  = \Exp(1)$. This implies $\overline{\cD}_\infty = 0.5$ and $\overline{\mathcal B}_\infty = 1$. 
\end{exm}

Similarly to~\eqref{eq:G-adjusted},~\eqref{eq:main-generator}, define
\begin{align}
\label{eq:new-operators}
\begin{split}
\tilde{\mathcal L}_{n, \nu}f(z) & := \cG_{n, \overline{\nu}}f(z) + \frac{\lambda_{\infty}\left(\overline{\nu}\right)}{n}\int_0^{\infty}\left[f(w) - f(z)\right]\mathcal B_{\infty, \overline{\nu}}(\mathrm{d}w) \\ & + \kappa_{\infty}(\overline{\nu}, z)\int_0^{\infty}\left[f(w) - f(z)\right]\,\nu(\mathrm{d}w);\\ \mbox{where} & \quad \cG_{n, y}f(z) =  \left[r - n\kappa_{\infty}(y, z)\overline{\cD}_{\infty}(y, z)\right]D_1f(z) + \frac12\sigma^2D^2f(z).
\end{split}
\end{align}
Similarly to~\eqref{eq:main-operator}, we  apply the current distribution $\nu$ to the generator $\tilde{\mathcal L}_{n, \nu}$ in~\eqref{eq:new-operators} and define
\begin{align}
\label{eq:new-operator}
\begin{split}
\tilde{\mathcal A}(n, \nu, f) := \Bigl[r  & - \kappa_{\infty}(\overline{\nu}, z)\overline{\cD}_{\infty}(\overline{\nu}, z) n_{\infty}\Bigr](\nu, D_1f)  + \frac12\sigma^2(\nu, D_2f) \\ & +  n^{-1}_{\infty}\Bigl[\lambda_{\infty}\left(\overline{\nu}\right)\left(\mathcal B_{\infty, \overline{\nu}}, f\right)  - \lambda_{\infty}\left(\overline{\nu}\right)\left(\nu, f\right)\Bigr]  \\ & +  (\nu, f) \left[\left(\nu, \kappa_{\infty}\left(\overline{\nu}, \cdot\right)\right) -  \left(\nu, \kappa_{\infty}\left(\overline{\nu}, \cdot\right)f\right)\right].
\end{split}
\end{align}
Consider the following McKean-Vlasov jump-diffusion $\tZ = (\tZ(t),\, t \ge 0)$,  with $\tm(t) = \mathbb E[ \tZ(t)]$, and $\tilde{\mu}^{(\infty)}_t \sim \tZ(t)$. Its generator at time $t$ is the version of the generator ~\eqref{eq:new-operators} (cf. ~\eqref{eq:main-generator}):
\begin{align}
\label{eq:new-generator}
\tilde{\mathcal L}_t  & := \tilde{\mathcal L}_{\mathcal N_{\infty}(t), \, \tilde{\mu}^{(\infty)}_t },
\end{align}
where the function $\mathcal N_{\infty} : \mathbb R_+ \to \mathbb R_+$ is the solution to the following linear first-order ODE:
\begin{equation}
\label{eq:ratio-limit}
\mathcal N'_{\infty}(t) = \lambda_{\infty}(\tm(t)) - \mathcal N_{\infty}(t)\bigl(\tilde{\mu}^{(\infty)}_t, \kappa_{\infty}(\cdot, \tm(t))\bigr),\quad \mathcal N_{\infty}(0) = 1.
\end{equation}
The role of $\mathcal N_{\infty}(t)$ is to scale the system size at time $t$ relative to its initial size $N$ at time $0$:
\[
\lim_{N\to \infty} \frac{\, \mathcal N_{N}(t) \,}{\,N\,}  = \mathcal N_{\infty}(t) .
\]
Solving this deterministic ODE \eqref{eq:ratio-limit} as follows:
\begin{align}
\label{eq:solution-ratio-limit}
\begin{split}
\mathcal N_{\infty}(t) & = \left[\mathcal K(t)\right]^{-1}\left[1 + \int_0^t\lambda_{\infty}(\tm(s))\mathcal K(s)\,\mathrm{d}s\right], \\
\mathcal K(t) & := \exp\left[\int_0^t(\tilde{\mu}^{(\infty)}_u, \kappa_{\infty}(\tm(u), \cdot))\,\mathrm{d}u\right],
\end{split}
\end{align}
 and plugging back into~\eqref{eq:new-generator}, we rewrite it as a McKean-Vlasov jump-diffusion, which is {\it time-inhomogeneous:}
Its parameters (specifically, the drift coefficient and the jump measure) depend on time $t$; in fact through $\mathcal N_{\infty}$, the $t$-dynamics depend on the whole history: $\tilde{\mu}^{(\infty)}_s,\, 0 \le s \le t$, rather than $\tilde{\mu}^{(\infty)}_t$ and $\tZ(t)$.

\smallskip

In the following formulae~\eqref{eq:new-diffusion}, ~\eqref{eq:new-drift-coefficient}, ~\eqref{eq:jump-measures-new}, $z \in \mathcal D[0, t]$, where $t$ is another argument. The argument $\tilde{\mu}^{(\infty)}$ represents a measure-valued function $(\tilde{\mu}^{(\infty)}_s,\, 0 \le s \le t)$. Its mean at time $t$ is denoted by $\tilde{\mathbf{m}}(t)$. The diffusion coefficient is very similar to the one in the first mean-field limit. (Slightly abusing the notation, we use $\sigma$ both for this coefficient and for the original volatility of each bank.)
\begin{equation}
\label{eq:new-diffusion}
\sigma\left(t, \tilde{\mu}^{(\infty)}, z\right) = \sigma z(t).
\end{equation}
The new drift coefficient is, however, different; it is given by
\begin{align}
\label{eq:new-drift-coefficient}
\begin{split}
g\left(t, \tilde{\mu}^{(\infty)}, z\right) & = z(t)\tilde{\psi}\left(\mathcal N_{\infty}(t), \tilde{\mathbf{m}}(t), z(t)\right), \\  \tilde{\psi}(n, y, v)  & := r - n\kappa_{\infty}(y, v)\overline{\cD}_{\infty}(y, v).
\end{split}
\end{align}
Thus, the counterpart $\mathcal G_{n, y}$ of $\mathcal G_y$ from~\eqref{eq:G-adjusted} can be written as
$$
\mathcal G_{n, y}f(v) = \tilde{\psi}(n, y, v)D_1 f(v) + \frac{\sigma^2}2D_2f(v).
$$
Finally,  the new jump measure is given by (compare with~\eqref{eq:jump-measures}):
\begin{equation}
\label{eq:jump-measures-new}
\tilde{\mu}_{t, z, \tilde{\mu}^{(\infty)}}(\mathrm{d}w) =  \mathcal N^{-1}_{\infty}(t)\lambda_{\infty}\left(\tilde{\mathbf{m}}(t)\right)\,\mathcal B_{\infty, \tilde{\mathbf{m}}(t)}(\mathrm{d}w) + \kappa_{\infty}(\tilde{\mathbf{m}}(t), z(t))\,\tilde{\mu}^{(\infty)}_t(\mathrm{d}w).
\end{equation}

\begin{rmk}
\label{rmk:existence-uniqueness}
The magnitude $\mathcal N_{\infty}(t)$, as a function of $(\tilde{\mu}^{(\infty)}_u,\, 0 \le u \le s)$, is bounded and Lipschitz with respect to 
$\mathcal W_p$ in $\mathcal P_p(B[0, t])$, with Lipschitz constant uniform in $t$. Then, the drift and diffusion coefficients from~\eqref{eq:new-diffusion},~\eqref{eq:new-drift-coefficient},~\eqref{eq:jump-measures-new} are Lipschitz with respect to 
$$
z = (z(u),\, 0 \le u \le t),\, (\tilde{\mu}^{(\infty)}_u,\, 0 \le u \le t),
$$
uniformly in $t$. This allows us to use the result of Remark~\ref{rmk:history}.
\end{rmk}

The following is a counterpart of our result in Theorem~\ref{thm:hydro}, with proof given in Section~\ref{sec:proof-new}.

\begin{thm} Fix $p > 1$. Assume initial empirical measures converge in $\mathcal W_p$:
$\mu^{(N)}_0 \Rightarrow \tilde\mu^{(\infty)}_0$. For every $T > 0$ and 
$q \in [1, p)$, under Assumptions~\ref{asmp:birth-measures-new},~\ref{asmp:birth-intensities-new},~\ref{asmp:default-measures-new},~\ref{asmp:default-intensities-new}, we have  convergence in law in the Skorohod space $\cD([0, T], \mathcal P_{q})$:
$$
\mu^{(N)} \Rightarrow \tilde{\mu}^{(\infty)}\ \mbox{as}\ N \to \infty
$$
where $\tilde{\mu}^{(\infty)}$ is a McKean-Vlasov-It\^o process with generator \eqref{eq:new-generator}.
\label{thm:hydro-new}
\end{thm}

For $q \in (1, p)$, the functional $\nu \mapsto (\nu, f_1)$ is continuous in $\mathcal W_q$. This immediately implies the following about the mean bank capital distributions:

\begin{cor} We have weak convergence of mean reserves as $N \to \infty$ in $\cD[0, T]$:
$$
\mathbf{m}_N(\cdot) \Rightarrow \tm(\cdot).
$$
\end{cor}

\subsection{Defaults independent of size} Under Assumptions~\ref{asmp:birth-measures-new},~\ref{asmp:birth-intensities-new},~\ref{asmp:default-measures-new},~\ref{asmp:default-intensities-new}, if $\kappa_{\infty}(y, x) = \kappa_{\infty}(y)$ and $\overline{\mathcal D}_{\infty}(y, x)$ are independent of $x$,  
the diffusion part of McKean-Vlasov equation for $\tZ$ is
\begin{equation}
\label{eq:McKean-new}
\md \tZ(t) = \bigl[r -  \overline{\cD}_{\infty}(\tm(t))\kappa_{\infty}(\tm(t))\mathcal N_{\infty}(t)\bigr]\tZ(t) \,\md t  + \sigma \tZ(t)\,\md B(t),
\end{equation}
with $\tZ$ is killed with rate $\mathcal N_{\infty}(t)\lambda_{\infty}(\tm(t))$, and resurrected according to the probability measure $\mathcal B_{\infty, \tm(t)}$. As before, only the first component in the jump measure~\eqref{eq:jump-measures-new} remains because $\kappa_\infty(y,x)$ does not depend on $x$. To solve \eqref{eq:McKean-new}
we first compute $\tm(t)$. Taking expectations, we obtain
\begin{align}
\label{eq:ODE-2}
\left\{ \begin{aligned}
\mathcal N'_{\infty}(t) & = \lambda_{\infty}(\tm(t)) - \mathcal N_{\infty}(t)\kappa_{\infty}(\tm(t)),
\\
\tm'(t) & = \left[r - \overline{\cD}_{\infty}(\tm(t))\kappa_{\infty}(\tm(t))\mathcal N_{\infty}(t)\right]\tm(t) +  \frac{\lambda_{\infty}(\tm(t))}{\mathcal N_{\infty}(t)}\left(\overline{\mathcal B}_\infty(\tm(t)) - \tm(t)\right).
\end{aligned} \right.
\end{align}
Assume this (deterministic) system~\eqref{eq:ODE-2} of ODEs has a unique solution $(\mathcal N_{\infty}, \tm)$ with the initial condition
$$
\tm(0) = \int_0^{\infty}\! x\mu^{(\infty)}_0(\md x),\qquad \mathcal N_{\infty}(0) = 1.
$$
Plug this in~\eqref{eq:McKean} to get that $\tZ$
is a geometric Brownian motion with time-dependent drift:
\begin{equation}
\label{eq:GBM-new}
\tZ(t) = \tZ(0)\exp\left[\int_0^t\left[r -  \overline{\cD}_{\infty}(\tm(s))\kappa_{\infty}(\tm(s))\mathcal N_{\infty}(s) - \sigma^2/2\right]\,\md s + \sigma B(t)\right],
\end{equation}
killed at rate $\lambda_{\infty}(\tm(t))$, and resurrected according to $\mathcal B_{\infty, \tm(t)}$. We revisit the case when all limiting parameters are constant. Then the system of differential equations~\eqref{eq:ODE-2} takes the form
\begin{align}
\label{eq:ODE-2-const}
\begin{split}
\mathcal N'_{\infty}(t) & = \lambda_{\infty} - \mathcal N_{\infty}(t)\kappa_{\infty},
\\
\tm'(t) & = r\tm(t) - \overline{\cD}_{\infty}\kappa_{\infty}\mathcal N_{\infty}(t) \tm(t) +{\lambda_{\infty}}{\mathcal N_{\infty}^{-1}(t)}\left(\overline{\mathcal B}_{\infty} - \tm(t)\right).
\end{split}
\end{align}
The first equation in~\eqref{eq:ODE-2-const} starting at $\mathcal N_{\infty}(0) = 1$ is solved as
$$
\mathcal N_{\infty}(t) = \frac{\lambda_{\infty}}{\kappa_{\infty}} - \left[\frac{\lambda_{\infty}}{\kappa_{\infty}} - 1\right]\exp\left(-\kappa_{\infty}t\right).
$$
The second equation of \eqref{eq:ODE-2-const}, which is also linear, can similarly be solved explicitly. As $t \to \infty$, $\mathcal N_{\infty}(t) \to \mathcal N_{\infty}(\infty) := \lambda_{\infty}/\kappa_{\infty}$. Therefore, we can find the long-term limit of $\tm(t)$ by plugging $\mathcal N_{\infty}(\infty)$ instead of $\mathcal N_{\infty}(t)$ into~\eqref{eq:ODE-2-const} and letting the right-hand side be equal to zero. This gives
$$
\tm(t) \to \frac{\kappa_{\infty}\overline{\mathcal B}_{\infty}}{\overline{\mathcal D}_{\infty}\lambda_{\infty} + \kappa_{\infty} - r},\qquad \mbox{as}\ t \to \infty.
$$
The left panel of Figure~\ref{fig:conv-N} illustrates such convergence to the mean field limit. We take $\lambda_N(n,s) = 0.2N$, $\kappa_N(n,s,x) = 0.1$, $\mathcal{B}_{n,s} \sim \Exp(1)$ and $\mathcal{D}_{N,n,s,x} \sim \Uni(0, 1/N)$, and $r=0.05$. Note that in the mean field limit $\lambda_\infty = 0.2$, $\kappa_\infty(x) = 0.1$, and $\overline{\mathcal{B}}_\infty = 1, \overline{\mathcal{D}}_\infty = 0.5$. Therefore,
$$
\lim_{t \to \infty} \tm(t) = \frac{0.1\cdot1}{0.5\cdot0.2 + 0.1 - 0.05} = \frac23.
$$
In Figure~\ref{fig:conv-N}, we initialize with $\mu_0^{(N)} \sim \Exp(2)$, so that $\tm(0) = 0.5$. We see that the solution $\tm(t)$ converges to its limiting value more slowly than in Figure~\ref{fig:conv}, which is not surprising since the ODE~\eqref{eq:ODE-2-const} contains $\mathcal N_{\infty}$, which is also not constant. Moreover, the more complicated ODE governing the evolution of $\tm$ leads to $t \mapsto \tm(t)$ being non-monotone in this particular setup.

\smallskip

Note the difference to the model in Section~\ref{sec:mean-field}. There, $\mathcal N(t)$ did not have a stationary distribution, since at level $\mathcal N(t) = n$, the birth rate was $0.2n$, larger than the total default rate $0.1 n$. As a result, $\mathcal N(t)$ was growing exponentially in $t$. In the present Section, the birth rate is $0.2N$ (constant with respect to $n$) and the death rate is $0.1 n$, so that $\mathcal N_N(t)$ is a constant-birth, linear-death process which has a stationary distribution of $N_\infty \sim \Poi (2N)$. Comparing to $\mathcal N_N(0) = N$, the relative ratio of $\E[ N_\infty]/\mathcal N_N(0) = 2$ matches the limit $\mathcal{N}_\infty(\infty) = \lambda_\infty/\kappa_\infty$. Similarly to Theorem~\ref{thm:individual} we have a propagation of chaos based on Theorem~\ref{thm:hydro-new}.

\begin{figure}[t]
\begin{center}
\includegraphics[height=3in,width=3in]{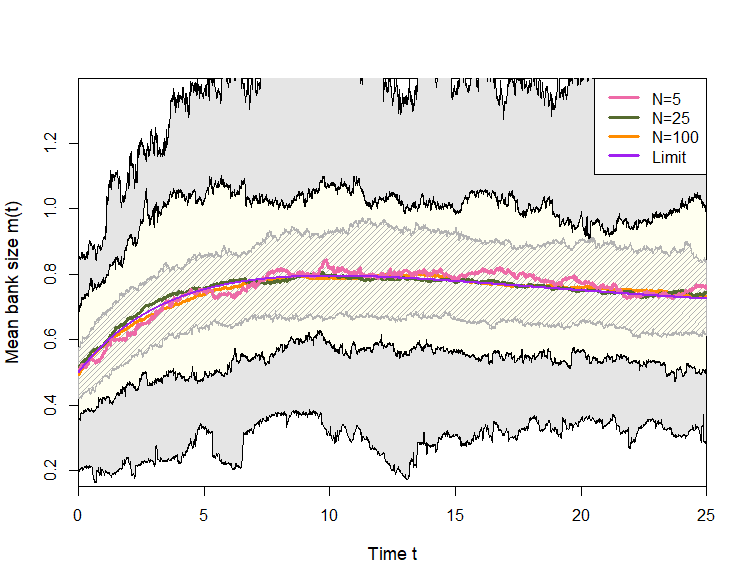}
\includegraphics[height=3in,width=3in]{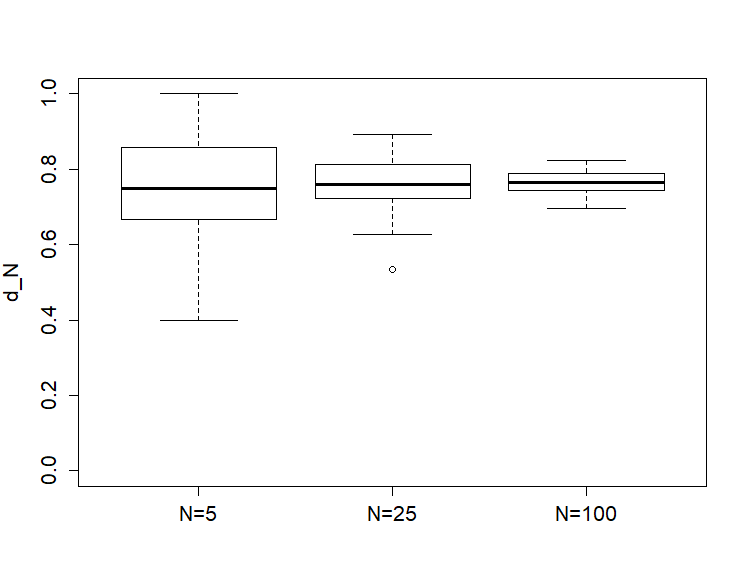}
\end{center}
\caption{Left panel: distribution of $m_N(t)$ on $t \in [0,25]$ and $N=5,25,100$ based on $100$ simulated trajectories of $X^{(N)}$. The initial distribution is $\mu^{(N)}_0 \sim \Exp(2)$, so $m_N(0) = 0.5$. The deterministic limit $\tm(\cdot)$ is also shown. Right panel: distribution of $d_N := (\mu^{(N)}_T, 1_{[0,D]})$, proportion of banks with reserves less than $D=1$ at $T=25$.}
\label{fig:conv-N}
\end{figure}

\begin{cor} We work under Assumptions~\ref{asmp:birth-measures-new},~\ref{asmp:birth-intensities-new},~\ref{asmp:default-measures-new},~\ref{asmp:default-intensities-new}. Assume $X_1^{(N)}(0)$ is deterministic for every $N \ge 1$, and $X_1^{(N)}(0) \to x_1$ as $N \to \infty$. As $N \to \infty$, $X_1^{(N)} \Rightarrow X^{(\infty)}_1$ weakly in $\cD[0, T]$, where $X^{(\infty)}_1$ is a solution to
\begin{align}\label{eq:x-infty-sde-2}
\md X^{(\infty)}_1(t) = \tilde{\psi}\bigl(t, X^{(\infty)}_1(t), \tm(t)\bigr)X^{(\infty)}_1(t)\,\md t + \sigma X^{(\infty)}_1(t)\,\md W(t),
\end{align}
starting from $x_1$, killed with rate $\kappa_{\infty}(\tm(t),X^{(\infty)}_1(t))$.
\label{thm:individual-new}
\end{cor}


\section{Proofs for Sections 2 and 3}\label{app:technical}

We start with the following three technical lemmata and their proofs. 

\begin{lemma}
If $\nu_n \to \nu_0$ in $\mathcal W_p$ for some $p \ge 1$, then  $(\nu_n, f) \to (\nu_0, f)$ for functions $f : \mathbb R_+ \to \mathbb R$ with the following property:
$$
C_f := \sup_{x \ge 1}\left[x^{-q}|f(x)|\right] < \infty\quad \mbox{for some}\quad 0 < q < p.
$$
In particular, the functional of taking the mean $\nu \mapsto \overline{\nu}$ is continuous in $\mathcal W_p$. 
\label{lemma:technical-convergence}
\end{lemma}

\begin{lemma} For $0 < q < p$ and any $C> 0$, the set $\{\nu \in \mathcal P_p \, :\,  (\nu, f_p) \le C\}$ is precompact in $\mathcal P_q$.
\label{lemma:precompact}
\end{lemma}

\begin{lemma} For every $\varepsilon \in (0, 1)$,  there exists a constant $C_{\varepsilon}$ such  that for every function $f \in \mathcal C^2_b$, and for $x > 0$, $z \in (0, 1 - \varepsilon)$, we have:
$$
\left|f(x(1-z)) - f(x) - zD_1f(x) \right| \le C_{\varepsilon}z^2\triplenorm{f}.
$$
\label{lemma:Taylor}
\end{lemma}

\subsection{Proof of Lemma~\ref{lemma:technical-convergence}} Let us take a sequence $(\xi_n)$ of random variables: $\xi_n \sim \nu_n$ for every $n$. By the Skorohod representation theorem, we can assume $\xi_n \to \xi_0$ a.s., as $n \to \infty$. We also have $\E \lvert \xi_n \rvert ^p \to \E \lvert \xi_0 \rvert ^p$, as $\nu_{n} \to \nu_{0}$ in the Wasserstein metric $\mathcal W_{p}$. Since $|f(x)|^{p/q} \le C_f^{p/q}x^{p}$ for $f \in \mathcal H_q$, we get $\sup_{n \ge 1}\E\left[ \lvert f(\xi_n) \rvert ^{p/q} \right] < \infty$. Thus for $p > q$ the family $(f(\xi_n))_{n \ge 1}$ is uniformly integrable, and hence, $(\nu_n, f) = \E f(\xi_n) \to \E f(\xi_0) = (\nu_0, f)$ as $n \to \infty$.

\subsection{Proof of Lemma~\ref{lemma:precompact}} Take a sequence $(\nu_n)_{n \ge 1}$ of measures in $\{ \nu \in \mathcal P_{p} : ( \nu, f_{p}) \le C\}$, and generate random variables $\xi_n \sim \nu_n$. Since $ \sup_{n \ge 1}\E \lvert \xi_n \rvert ^p \le C < \infty,$ the sequence $(\xi_n)_{n \ge 1}$ is tight. Extract a weakly convergent subsequence; without loss of generality, we assume this sequence itself converges weakly to some random variable $\xi_0$. By the Skorohod representation theorem, we can assume $\xi_n \to \xi_0$ a.s. Moreover, the sequence $( \lvert \xi_n \rvert^q)_{n \ge 1}$ is uniformly integrable for $q < p$, since $\E[( \lvert \xi_n \rvert^q)^{p/q}] \le C$. Therefore, $\E \lvert \xi_n \rvert ^q \to \E \lvert \xi_0 \rvert ^q$, and hence, $\xi_n \to \xi_0$ in  $\mathcal W_q$. 

\subsection{Proof of Lemma~\ref{lemma:Taylor}} Take a function $g(x) := f(e^x)$. Then
$$
g'(x) = e^xf'(e^x) = (D_1f)(e^x),\quad g''(x) = e^xf'(e^x) + e^{2x}f''(e^x) = (D_1f + D_2f)(e^x).
$$
We can rewrite it for some $y \in [\ln x + \ln(1 - z), \ln x]$:
\begin{align*}
f(x(1-z)) & - f(x) - zD_1f(x) \\ & = g(\ln x + \ln(1-z)) - g(\ln x) - zg'(\ln x) \\ & = g(\ln x + \ln(1-z)) - g(\ln x) + \ln(1-z)g'(\ln x) - (\ln(1-z) + z)g'(\ln x) \\ & =
\frac12\ln^2(1-z)g''(y) - (\ln(1-z) + z)g'(\ln x).
\end{align*}
There exists a $C_{\varepsilon}$ such that for $z \in [0, 1 - \varepsilon]$,
$$
\frac12\ln^2(1 - z) \le C_{\varepsilon}z^2,\quad \ln(1-z) + z \le C_{\varepsilon}z^2.
$$
It suffices to note that for all $y \in \mathbb R$, we have: $|g''(y)| \le \triplenorm{f}$; and for all $x > 0$, we have: $|g'(\ln x)| \le \triplenorm{f}$. This completes the proof.

\subsection{Proof of Lemma~\ref{lemma:totally-irreducible}}\label{sec:Lemma31}
We need to establish~\eqref{eq:totally-irreducible}, i.e.,~$P^t(\bx,A) > 0 \ \forall t  > 0$. Assume first that $A \subseteq (0, \infty)^N$ for $N = \mathfrak n(\bx) \ge 1$; that is, the target set $A$ lies on the same level as the initial point $\mathbf{x}$. Observe that the intensities of births and defaults of banks are locally bounded on $(0, \infty)^N$ as long as there are $N$ banks in the system; therefore, with positive probability there are $N$ banks at every time $s \in [0, t]$, and the process $X$ behaves as the solution to a certain stochastic differential equation on $(0, \infty)^N$ with a nonsingular covariance matrix. But such processes have the positivity property.

\smallskip

If $N = \mathfrak n(\bx) = 0$, and $A = \{\varnothing\}$, then $P^t(\mathbf{x}, A) = e^{-\lambda_0t} > 0$: This is the probability that, starting with an empty system, no banks emerged during time $[0, t]$.

\smallskip

Assume now that $A \subseteq (0, \infty)^M$ for $M \ne N = \mathfrak n(\bx) \ge 1$. Then with positive probability we have: $\mathfrak n(X(s)) = M$ for some $s \in [0, t)$, since
the rates of birth and default are everywhere positive. Let $\tau'$ be the first moment of hitting level $M$:
$$
\tau' := \inf\{s \ge 0\mid \mathfrak n(X(s)) = M\}.
$$
Observe that the integral of a positive function over a set of positive measure is positive. Applying this and conditioning on $\tau'$ and $X(\tau')$, by the Markov property of $X$ we get:
$$
\PP(X(t) \in A) = \int_0^t\PP(\tau' \in \md s,\, X(\tau') \in \mathrm{d} \mathbf{y})\, \PP\left(X(t-s) \in A\mid X(0) = \mathbf{y}\right) > 0.
$$
This completes the proof of~\eqref{eq:totally-irreducible} for subsets $A$ which are on one level of $\mathcal X$. Any general set $A \subseteq \mathcal X$ can be split into its subsets, at least one of which has positive Lebesgue measure.

\subsection{Proof of Theorem~\ref{thm:conservative}} \label{sec:proofofconserva}
The statement of Theorem~\ref{thm:conservative} then follows from Lemmata~\ref{lemma:totally-irreducible},~\ref{lemma:Feller}, and the classic results of \cite{MT1993a, MT1993b}, together with \cite[Proposition 2.2, Lemma 2.3]{MyOwn10}. In fact, since the last two terms in \eqref{eq:Lyapunov-result} are non-positive, the condition in Theorem~\ref{thm:conservative} is effectively about the term due to births
$\lambda_{\mathfrak n(\bx)}(\mathfrak{s}(\bx))\left[\overline{\mathcal B}(\mathfrak n(\bx), \mathfrak{s}(\bx)) + 1\right], $
growing at most linearly in $\mathfrak{n}(x)$ and $ \mathfrak{s}(x)$.

\subsection{Proof of Theorem~\ref{thm:stability}} Without loss of generality, assume $c_1 = c_2 = c$ by taking the smaller one among $c_1$ and $c_2$. Compare~\eqref{eq:Lyapunov-result} with~\eqref{eq:1205} and observe that for every $\varepsilon \in (0, 1)$
$$
\mathfrak{L} V_0(\bx) = \varphi(\bx) \le -cV_0(\bx) + c_3 \le - (1-\varepsilon) cV_0(\bx) +  c_3 \cdot 1_{\mathcal K}(\bx),\ \mathbf{x} \in \mathcal X,
$$
where $\mathcal K :=\{ \bx \in \mathcal X : V_{0} (\bx) \le c_{3}/(\varepsilon c) \} $ is a compact subset of $\mathcal X$. The bound \eqref{eq:exp-ergodic-V} follows from Lemmata~\ref{lemma:totally-irreducible},~\ref{lemma:Feller}, and from the theory of Lyapunov functions, e.g. \cite{MT1993a, MT1993b, MyOwn10}

\subsection{Proof of Theorem~\ref{thm:explicit-rate}} Take two copies $X^{(1)} = (X^{(1)}(t),\, t \ge 0)$ and $X^{(2)} = (X^{(2)}(t),\, t \ge 0)$ of this process, starting from $X^{(1)}(0) = \bx_1$ and $X^{(2)}(0) = \bx_2$. The idea is to couple them when the dimension-counting processes $N^{(1)}(\cdot) = \mathfrak{n}(X^{(1)}(\cdot))$ and $N^{(2)}(\cdot) = \mathfrak{n}(X^{(2)}(\cdot))$  meet at $0$. Then the original processes $X^{(1)}$ and $X^{(2)}$ meet at $\varnothing$. Define this coupling time $\tau$:
\begin{equation}
\label{eq:coupling-time}
\tau := \inf\{t \ge 0\mid N^{(1)}(t) = N^{(2)}(t) = 0\}.
\end{equation}
By classic Lindvall's inequality, the total variation distance from~\eqref{eq:TV} between $P^t(\bx_1, \cdot)$ and $P^t(\bx_2, \cdot)$ is less than or equal to $2\,\mathbb P(\tau \ge t)$. Next, compare these dimension-counting processes $N^{(i)}$ with birth-death processes: $ N^{(i)}(t) = \mathfrak n(X^{(i)}(t)) \le \hat{N}^{(i)}(t)$,
$ i = 1, 2$, $\, t \ge 0$, where $\hat{N}^{(i)} = (\hat{N}^{(i)}(t),\, t \ge 0)$ is a birth-death process with birth intensity $\lambda^*_n$ and death intensity $n\kappa^*_n$ at site $n \in \mathbb Z_+$, starting from $\hat{N}^{(i)}(0) := N^{(i)}(0)$, $i = 1, 2$. Similarly to \cite{LMT1996, MyOwn12}, we find that the moment $\tau$ satisfies the following estimate:
\begin{equation}
\label{eq:mgf-estimate}
\mathbb E [e^{\alpha\tau}] \le \hat{V}(\hat{N}^{(1)}(0)\vee \hat{N}^{(2)}(0)) = \tilde{ V}(\bx_1)\vee \tilde{V}(\bx_2).
\end{equation}
The coupling time~\eqref{eq:coupling-time} for the processes $\hat{N}^{(1)}$ and $\hat{N}^{(2)}$ is also a coupling time for the processes $X^{(1)}$ and $X^{(2)}$. The rest of the proof is similar to \cite[Theorem 2.2]{LMT1996}, \cite[Section 5]{MyOwn12}.

\section{Proof of Theorem~\ref{thm:hydro}}\label{sec:theorem41}

\subsection{Overview of the proof} Recall the definition of $\mathcal E_f$ from~\eqref{eq:empirical-function}. It\^o's formula applied  to $\mathcal E_f(X^{(N)}(t)) =(\mu^{(N)}_t, f)$ for some function $f \in \mathcal C_b^2$ reads as:
\begin{equation}
\label{eq:Ito}
\mathcal E_{f}\left(X^{(N)}(t)\right) = \mathcal E_{f}\left(X^{(N)}(0)\right) + \int_0^t\cL\mathcal E_{f}\bigl(X^{(N)}(s)\bigr)\,\mathrm{d}s + \mathcal M_N^f(t),
\end{equation}
where $(\mathcal M_N^f(t),\, t \ge 0)$ is a real-valued rcll local martingale. Between jumps (while the number of banks  $\mathfrak{n}( X^{(N)}(t) ) = \mathcal N_{N}(t)$ stays constant), the local martingale $\mathcal M_N^f$ is given by
$$
\mathrm{d}\mathcal M_N^f(t) := \frac{\sigma}{\mathcal N_N(t)}\sum\limits_{i=1}^{\mathcal{N}_N(t)}(D_1f)\bigl(X_i^{(N)}(t)\bigr)\,\mathrm{d}W_i(t).
$$
First, let us state the main convergence lemma, which makes the analytical crux of the proof.

\begin{lemma}
Take a function $f \in \mathcal C^2_b$. Recalling~\eqref{eq: empiricaldist}, take a sequence $(\mathbf{x}^{(k)})_{k \ge 1}$ in $\mathfrak X$ with
\begin{equation}
\label{eq:standard-sequence}
\mathfrak n\left(\mathbf{x}^{(k)}\right) = k\quad\mbox{and}\quad \mu_{\mathbf{x}^{(k)}} \to \nu\quad \mbox{in}\ \mathcal W_p \quad\text{as} \quad k \to \infty.
\end{equation}
 Then we have the following convergence of means and generators, as $N\to \infty$:
\begin{align}
\label{eq:fundamental-convergence}
\overline{\mathbf{x}}^{(k)} \to \overline{\nu};\quad \mbox{and}\quad 
\mathfrak{L}\mathcal E_f\left(\mathbf{x}^{(k)}\right) \to \mathcal A(\nu, f) \quad \text{from} \quad\eqref{eq:main-operator}.
\end{align}
\label{lemma:fundamental-convergence}
\end{lemma}

The following technical estimate is used repeatedly in the subsequent proofs.

\begin{lemma} For a constant $C$ depending on the parameters, we have
$$
\left|\mathfrak{L}\mathcal E_f(\mathbf{x})\right| \le C\left[1 + \overline{\mathbf{x}}\right]\cdot\triplenorm{f},\quad f \in \mathcal C^2_b,\quad \mathbf{x} \in \mathcal X\setminus\{\varnothing\}.
$$
\label{lemma:preliminary-estimate}
\end{lemma}

We next show that the term $\mathcal M_N^f$ tends to zero. The rough idea is as follows: Since jump sizes tend to zero, the process converges to a continuous limit. Since the quadratic variation converges to zero, the limit is a continuous martingale with zero quadratic variation, which implies that the limit itself is identically zero.  To formalize this argument and apply it to our more complicated situation, we state and prove the following series of lemmata.

\begin{lemma} For every $T > 0$, $r > 1$, and $f \in \mathcal C^2_b$,
$$
\E \Big[ \sup\limits_{0 \le s \le T}\bigl[\mathcal M_N^f(s)\bigr]^r \Big] \to 0,\quad\mbox{as}\quad N \to \infty.
$$
\label{lemma:mgle-to-0}
\end{lemma}


\begin{lemma} For every $T > 0$, there exists a constant $C_{T} > 0$ such that
$$
\E \Big[ \sup\limits_{t \in [0, T]}\bigl(\mu^{(N)}_t, f_p\bigr) \Big] \le C_{T}\quad\mbox{for all}\quad N \ge 1.
$$
\label{lemma:bdd-moment}
\end{lemma}


\begin{lemma} The sequence $(\mu^{(N)}_t,\, 0 \le t \le T)_{N \ge 1}$ is tight in $\cD([0, T], \mathcal P_q)$ for every $q< p$.
\label{lemma:tight}
\end{lemma}

Assume we already proved Lemmata \ref{lemma:fundamental-convergence},~\ref{lemma:mgle-to-0},~\ref{lemma:bdd-moment},~\ref{lemma:tight}. Let us complete the proof of Theorem~\ref{thm:hydro}. In light of Lemma~\ref{lemma:tight}, it suffices to show the following statement: For $q \in [1,  p)$, every weak limit point $\bigl(\mu^{(\infty)}_t,\, t \in [0, T]\bigr)$ of $\bigl(\mu^{(N)}_t,\, 0 \le t \le T\bigr)_{N \ge 1}$ in $\cD([0, T], \mathcal P_q)$ is governed by the McKean-Vlasov equation~\eqref{eq:McKean}.
Indeed, for any function $f \in \mathcal C^2_b$, we can rewrite~\eqref{eq:Ito} as follows:
\begin{equation}
\label{eq:again-Ito}
\bigl(\mu^{(N)}_t, f\bigr) = \bigl(\mu^{(N)}_0, f\bigr) + \int_0^t \cL \cE_f\left(X^{(N)}(s)\right)\,\mathrm{d}s + \mathcal M^f_N(t).
\end{equation}
Lettting $N \to \infty$ in~\eqref{eq:again-Ito} with Lemmata  \ref{lemma:fundamental-convergence}, ~\ref{lemma:mgle-to-0} and  \ref{lemma:bdd-moment}, we have that the last term vanishes while the key middle term converges to $\mathcal A(\mu^{(\infty)}_s, f)$. Overall we thus obtain that the limit obeys
$$
\bigl(\mu^{(\infty)}_t, f\bigr) = \bigl(\mu^{(\infty)}_0, f\bigr) + \int_0^t\mathcal A(\mu^{(\infty)}_s, f)\,\mathrm{d}s.
$$
Since this holds true for all $f \in \mathcal C^2_b$, then, as explained in subsection 4.1, this is the equivalent definition of the McKean-Vlasov jump-diffusion. This completes the proof of Theorem~\ref{thm:hydro}.

\subsection{Proof of Lemma~\ref{lemma:fundamental-convergence}} Convergence of means follows from Lemma~\ref{lemma:technical-convergence}. Now, let us show the second statement in~\eqref{eq:fundamental-convergence}.  Apply the generator $\mathfrak{L}$ from~\eqref{eq:generator} to $\mathcal E_f$ from~\eqref{eq:empirical-function}, for $f \in \mathcal C^2$, with the argument $\mathbf{x} = \bigl(x_1, \ldots, x_{\mathfrak{n}(\mathbf{x})}\bigr) \ne \varnothing$. At first, we just do calculations of the generator, and only afterwards we plug in $\mathbf{x}^{(k)}$ instead of $\mathbf{x}$. Corresponding to the three lines in the right-hand side of \eqref{eq:generator} we shall use the shorthand $\mathfrak{L}\, \cE_f = I_1 + I_2 + I_3$. The first term $I_1$ involving the diffusion operator is calculated as follows:
\begin{equation*}
\frac{\partial\mathcal E_f}{\partial x_i}(\bx) = \frac1{\mathfrak n(\bx)}f'(x_i),\quad \frac{\partial^2\mathcal E_f}{\partial x_i^2}(\bx) = \frac1{\mathfrak n(\bx)}f''(x_i),\quad i = 1, \ldots, \mathfrak n(\mathbf x),
\end{equation*}
which leads to
\begin{equation}
\label{eq:term-1}
I_1 = \frac1{\mathfrak n(\bx)}\sum\limits_{i=1}^{\mathfrak n(\bx)}\Bigl[rx_if'(x_i) + \frac{\sigma^2}2x_i^2f''(x_i)\Bigr] = \frac1{\mathfrak n(\bx)}\sum\limits_{i=1}^{\mathfrak n(\bx)}\cG f(x_i) = (\mu_{\mathbf{x}}, \cG f).
\end{equation}
Next, the second term with the birth rates  is equal to
\begin{align}
\label{eq:term-2}
\begin{split}
I_2 & =  \lambda_{\mathfrak n(x)}(\mathfrak{s}(\bx))\int_0^{\infty}\Bigl[\frac1{\mathfrak n(\bx)+1}\Bigl(\sum\limits_{i=1}^{\mathfrak n(\bx)}f(x_i) + f(y)\Bigr) - \frac1{\mathfrak n(\bx)}\sum\limits_{i=1}^{\mathfrak n(\bx)}f(x_i)\Bigr]\,\mathcal B_{\mathfrak n(\bx), \mathfrak s(\bx)}(\mathrm{d} y) \\ & =
-\frac1{\mathfrak n(\mathbf x)(1 + \mathfrak n(\mathbf x))}\lambda_{\mathfrak n(\bx)}( \mathfrak s(\bx))\sum\limits_{i=1}^{\mathfrak n(\mathbf x)}f(x_i) + \frac{\lambda_{\mathfrak n(\mathbf x)}(\mathfrak s(\mathbf x))}{\mathfrak n(\mathbf x) + 1}\int_0^{\infty}f(y)\,\mathcal B_{\mathfrak n(\mathbf x), \mathfrak s(\mathbf x)}(\mathrm{d} y) \\ & =
- \frac{\lambda_{\mathfrak n(\bx)}(\mathfrak s(\bx))}{\mathfrak n(\mathbf x)+1}\left(\mu_{\mathbf{x}}, f\right) +
\frac{\lambda_{\mathfrak n(\mathbf x)}(\mathfrak s(\mathbf x))}{\mathfrak n(\mathbf x) + 1}\int_0^{\infty}f(y)\,\mathcal B_{\mathfrak n(\mathbf x), \mathfrak s(\mathbf x)}(\mathrm{d} y).
\end{split}
\end{align}
Finally, the third term is
\begin{align*}
I_3 = \sum\limits_{i=1}^{\mathfrak n(\mathbf x)}&\kappa_{\mathfrak n(\mathbf x)}(\mathfrak s(\mathbf x), x_i) \frac1{\mathfrak n(\mathbf x) -1 }\sum\limits_{j \neq i }^{\mathfrak n(\mathbf x)} \int_0^1 f(x_j(1-z_j)) \,\cD_{\mathfrak n(\mathbf x), \mathfrak s(\mathbf x), x_i}(\mathrm{d} z_j) \\  & - \frac1{\mathfrak n(\mathbf x)}\sum\limits_{i=1}^{\mathfrak n(\mathbf x)}\kappa_{\mathfrak n(\mathbf x)}(\mathfrak s(\mathbf x), x_i)\sum\limits_{j=1}^{\mathfrak n(\mathbf x)} \int_0^{1} f(x_j)\,\cD_{\mathfrak n(\mathbf x), \mathfrak s(\mathbf x), x_i}(\mathrm{d}z_j),
\end{align*}
which we re-arrange as
\begin{align}
\label{eq:term-3}
\begin{split}
I_3 &= \frac1{\mathfrak n(\mathbf x) }\sum\limits_{i=1}^{\mathfrak n(\mathbf x)}\kappa_{\mathfrak n(\mathbf x)}(\mathfrak s(\mathbf x), x_i)\sum\limits_{j  =1 }^{\mathfrak n(\mathbf x)} \int_0^{1} \left[f(x_j(1-z_j)) - f(x_j)\right]\,\cD_{\mathfrak n(\mathbf x), \mathfrak s(\mathbf x), x_i}(\mathrm{d} z_j) \\ & + \frac1{\mathfrak n(\mathbf x)(\mathfrak n(\mathbf x)-1)}\sum\limits_{i=1}^{\mathfrak n(\mathbf x)}\kappa_{\mathfrak n(\mathbf x)}(\mathfrak s(\mathbf x), x_i)\sum\limits_{j=1}^{\mathfrak n(\mathbf x)} \int_0^{1} f(x_j(1 -z_j))\,\cD_{\mathfrak n(\mathbf x), \mathfrak s(\mathbf x), x_i}(\mathrm{d}z_j) \\ & - \frac1{\mathfrak n(\mathbf x)-1}\sum\limits_{i=1}^{\mathfrak n(\mathbf x)}\kappa_{\mathfrak n(\mathbf x)}(\mathfrak s(\mathbf x), x_i)\int_0^{1}f(x_i(1-z_i))\,\cD_{\mathfrak n(\mathbf x), \mathfrak s(\mathbf x),x_i}(\mathrm{d} z_i) =: I_{3,1} + I_{3,2} + I_{3,3} .
\end{split}
\end{align}
Now, substitute the following sequence in the formulae above:
\begin{equation}
\label{eq:substitution}
\mathbf{x} := \mathbf{x}^{(k)},\quad \mathfrak{n}\bigl(\mathbf{x}^{(k)}\bigr) = k,\quad s_k := \mathfrak s\left(\mathbf x^{(k)}\right).
\end{equation}
The first term of $\mathfrak{L}\mathcal E_f\left(\mathbf{x}^{(k)}\right)$, given  in~\eqref{eq:term-1}, converges as $k \to \infty$ as follows:
\begin{equation}
\label{eq:first-conv}
I_1 = (\mu_{\mathbf x^{(k)}}, \mathcal{G} f) \to (\nu, \cG_{\overline{\nu}}f).
\end{equation}
This follows from the observation that $\cG f \in \mathcal H_q$, and Lemma~\ref{lemma:technical-convergence}. Next, we get  convergence of the second term $I_2$  given in~\eqref{eq:term-2}:
\begin{equation}
\label{eq:second-conv}
I_2 \to \lambda_{\infty}\left(\overline{\nu}\right)\left(\mathcal B_{\infty, \overline{\nu}}, f\right)  - \lambda_{\infty}\left(\overline{\nu}\right)\left(\nu, f\right)
\end{equation}
from Assumptions~\ref{asmp:birth-measures} and~\ref{asmp:birth-intensities}, together with the observation that $f \in \mathcal C^2_b$, and another application of Lemma~\ref{lemma:technical-convergence}. Finally, let us show convergence of $I_3$ from~\eqref{eq:term-3}, i.e.,
\begin{equation}
\label{eq:third-conv}
I_3 \to  -\left(\nu, \kappa_{\infty}\left(\overline{\nu}, \cdot\right)f\right)  + (\nu, f) \left(\nu, \kappa_{\infty}\left(\overline{\nu}, \cdot\right)\right), \quad \text{ as } \quad k \to \infty .
\end{equation}
The first term $I_{3,1}$ in~\eqref{eq:term-3} can be expressed, using Lemma~\ref{lemma:Taylor}:
\begin{equation}
\label{eq:first-904}
I_{3,1} = -\frac1{k}\sum\limits_{i=1}^{k}\kappa_{k}(s_k, x^{(k)}_i)\cdot\sum\limits_{j=1}^{k}D_1f(x^{(k)}_j)\int_0^{1}z\,\cD_{k, s_k, x^{(k)}_i}(\mathrm{d} z) + \delta_k =: J_k + \delta_k,
\end{equation}
where the residual $\delta_k$ for $k \ge n_0$ can be estimated as
$$
|\delta_k | \le C_{\varepsilon_0}\triplenorm{f}\cdot \frac1k \sum\limits_{i=1}^{k}\kappa_{k}(s_k, x^{(k)}_i)\cdot k\cdot \int^{1}_{0}z^2 \mathcal D_{k, s_{k}, x_{i}^{(k)}} ({\mathrm d} z ).
$$
By Remark~\ref{rmk:technical-default-measures} and Assumption~\ref{asmp:default-intensities},
\begin{equation}
\label{eq:residual-estimate}
|\delta_k| \le  k^{-1}C_{\varepsilon_0}C_{\kappa}C_{\mathcal D, 2}\cdot\triplenorm{f}.
\end{equation}
Finally, the main term $J_k$ in~\eqref{eq:first-904} of $I_{3,1}$ can be written as
\begin{align}\notag
J_k &= -\frac{1}{k}\sum\limits_{i=1}^k\kappa_{k}(s_k, x^{(k)}_i)\cdot \frac{1}{k}\sum\limits_{j=1}^{k}kD_1f(x^{(k)}_j)\overline{\cD}(k, s_k, x^{(k)}_i) \\ & = -\left(\mu_{\mathbf{x}^{(k)}}, k\overline{\cD}(k, s_k, \cdot)\kappa_{k}(s_k, \cdot)\right)\left(\mu_{\mathbf{x}^{(k)}}, D_1f\right).
\label{eq:first-908}
\end{align}
As $k \to \infty$, the expression~\eqref{eq:first-908} tends to
\begin{equation}
\label{eq:J-n-conv}
\lim_{k \to \infty} J_k = -\left(\nu, \kappa_{\infty}(\overline{\nu}, \cdot)\overline{\cD}_{\infty}(\overline{\nu}, \cdot)\right)\left(\nu, D_1f\right).
\end{equation}
From~\eqref{eq:first-904},~\eqref{eq:residual-estimate}, and~\eqref{eq:J-n-conv}, we get
\begin{equation}
\label{eq:I31-conv}
I_{3, 1} \to -\left(\nu, \kappa_{\infty}(\overline{\nu}, \cdot)\overline{\cD}_{\infty}(\overline{\nu}, \cdot)\right)\left(\nu, D_1f\right),
\end{equation}
which becomes the mean-field drift term in \eqref{eq:G-adjusted}. Similarly, we can show that the second and third terms  $I_{3,2}, I_{3,3}$ in~\eqref{eq:term-3} converge respectively to:
\begin{equation}
\label{eq:two-conv}
\lim_{k \to \infty} I_{3,2} =  +(\nu, f) \left(\nu, \kappa_{\infty}\left(\overline{\nu}, \cdot\right)\right)\quad \mbox{and}\quad \lim\limits_{k \to \infty} I_{3,3} =  -\left(\nu, \kappa_{\infty}\left(\overline{\nu}, \cdot\right)f\right)\quad
\end{equation}
Let us show this for $I_{3, 2}$; the proof for $I_{3, 3}$ is similar.  It follows from Assumption~\ref{asmp:default-measures} that the default contagion measures $\mathcal D_{\cdot, \cdot, \cdot}$ converge to $\delta_0$ (delta mass measure at zero) uniformly in $\mathcal W_p$. Because $f \in \mathcal C^2_b$, we have the following convergence as $k \to \infty$, uniformly over $j$:
$$
\int_0^{1} f(x_j(1 -z_j)) \,\cD_{k, s_k, x^{(k)}_i}(\mathrm{d}z_j)  \to f(x_j),
$$
which together with the assumption~\eqref{eq:standard-sequence} yields
\begin{equation}
\label{eq:conv-aux-1}
\frac1{k}\sum\limits_{j=1}^{k} \int_0^{1} f(x_j(1 -z_j)) \,\cD_{k, s_k, x^{(k)}_i}(\mathrm{d}z_j) \to (\nu, f),
\end{equation}
 as $k \to \infty$. Finally, by uniform boundedness of $\kappa_{\cdot}(\cdot, \cdot)$ together with~\eqref{eq:standard-sequence}, we get:
\begin{equation}
\label{eq:conv-aux-2}
\frac1k\sum\limits_{i=1}^k\kappa_k\bigl(s_k, x_i^{(k)}\bigr) \to \bigl(\nu, \kappa_\infty\bigl(\overline{\nu}, \cdot\bigr)\bigr).
\end{equation}
Combined,~\eqref{eq:I31-conv} and~\eqref{eq:two-conv} complete the proof of~\eqref{eq:third-conv}, and of Lemma~\ref{lemma:fundamental-convergence}.

\subsection{Proof of Lemma~\ref{lemma:preliminary-estimate}} From Assumptions~\ref{asmp:birth-measures},~\ref{asmp:birth-intensities},~\ref{asmp:default-measures},~\ref{asmp:default-intensities}, we estimate separately each term for $f$ in $\mathfrak{L}\,\mathcal E_f(x) = I_1 + I_2 +I_3$, given in~\eqref{eq:term-1}-\eqref{eq:term-3}. The first term from~\eqref{eq:term-1} is estimated as:
$$
|I_1| \le r\norm{D_1f} + \frac{\sigma^2}2\norm{D_2f}.
$$
For the second term in~\eqref{eq:term-2}, from Assumption~\ref{asmp:birth-intensities}, we get:
$$
|I_2| \le 2C_{\lambda}\norm{f}\left(1 + \overline{\bx}\right).
$$
Finally, consider the third term in~\eqref{eq:term-3}. Via Assumptions~\ref{asmp:default-measures} and~\ref{asmp:default-intensities}, similarly to the proof of Lemma~\ref{lemma:fundamental-convergence}, this term is estimated as
$$
|I_3| \le 2C_{\kappa}\norm{f} + C_{\kappa}C_{\cD, 1}\norm{D_1f} + C_{\kappa}C_{\cD, 2}\norm{D_2f} + C_{\kappa}C_{\varepsilon_0}C_{\cD, 2}\cdot\triplenorm{f},
$$
where $C_{\cD, p}$ was defined in~\eqref{eq:D-const}. Combining these estimates, we complete the proof of Lemma~\ref{lemma:preliminary-estimate}.

\subsection{Estimation of the number of banks from above and below} These results will be needed for the proof of Lemma~\ref{lemma:bdd-moment} and Lemma~\ref{lemma:tight}. Define the minimal and maximal number of banks in the system $X^{(N)}$ on time horizon $[0, T]$:
$$
\mathfrak M^-_N(T) := \min\limits_{0 \le t \le T}\mathcal N_N(t),\quad \mathfrak M^+_N(T) := \max\limits_{0 \le t \le T}\mathcal N_N(t).
$$
We start by estimating $\mathfrak M^-_N(T)$ from below. First, we claim that $\mathfrak M^-_N(T)$ stochastically dominates a Binomial random variable $\xi_{N}$  with parameters
$ \xi_{N} \sim \mathrm{Bin}(N, e^{-C_{\kappa}T})$  
with mean $\bar{\xi} := N e^{-C_{\kappa}T}$.
Indeed, $\mathcal N_N(0) = N$, and the default intensities are uniformly bounded from above by the constant $C_{\kappa}$. Then assume there is no birth of new banks, and all default intensities are exactly $C_{\kappa}$ on $[0,T]$ as an extreme case. This makes the number of banks at $T$ fewer  than for our original system $X^{(N)}$ and  distributed as the binomial random variable $\xi_{N}$. 
The latter tends to infinity in law: $ \mathfrak M^-_N(T) \to \infty$ as $N\to \infty$ from Chernov's inequality 
\begin{equation}
\label{eq:Chernov}
\PP\big (\xi_{N} \le \bar{\xi}/2\big) \le \exp\left(-\bar{\xi}/8\right) 
\end{equation} 
and from it, we get 
the following estimate: there exists a constant $C_{\mathfrak M}$ such that 
\begin{equation}
\label{eq:r-moment}
\E\left[(\mathfrak M^-_N(T)\vee 1)^{-r}\right] \le \E\left[ (\xi_{N}\vee 1)^{-r} \right]\le \left(\bar{\xi}/2\right)^{-r} + \exp\left(-\bar{\xi}/8\right) \le C_{\mathfrak M}N^{-r}, \quad r > 0 . 
\end{equation}

Now, let us estimate the maximal number of banks from above. Consider a pure birth process $\beta_N = (\beta_N(t),\, t \ge 0)$ on $\{1, 2, \ldots\}$ starting from $\beta_N(0) = N$, such that the intensity of births from level $n$ to $n+1$ is equal to $C_{\lambda}n$. Recall the estimate $\lambda_N(s) \le C_{\lambda}N$ in Assumption~\ref{asmp:birth-intensities}. We have the following observation: If there are no defaults, then $\mathcal{N}_N(t)$ is dominated by the above birth process $\beta_N$: $\mathcal N_N(t) \le \beta_N(T)$, where $\beta_N(0) = N$. At the same time,
$$
d\E[\beta_N(t)]/dt = C_{\lambda}\E[\beta_N(t)],\quad \mbox{which implies}\quad \E [\beta_N(T)] = e^{C_{\lambda}T}N.
$$
Therefore, for every $N \ge 1$,
\begin{equation}
\label{eq:max-exp}
\E \left[\mathfrak M^+_N(T)\right]  \le e^{C_{\lambda}T}N.
\end{equation}
What is more, we can estimate the second moment: The generator of $N^{-1}\beta_N$ is
$$
\mathcal L_{N}f(x) = Nx\left(f(x + N^{-1}) - f(x)\right).
$$
Applying this to function $f := f_2$, we get:
$$
\mathcal L_{N}f_2(x) = 2x^2 + xN^{-1}.
$$
If $m_{N}(t) := N^{-2}\E \beta_N^2(t)$, we can write Kolmogorov equations:
$$
m'_N(t) = \E [ \mathcal L_Nf_2(N^{-1}\beta_N(t))] = 2m_N(t) + N^{-1}e^{C_{\lambda}T},\quad m_N(0) = 1.
$$
Solving this, it is easy to see that $\sup_{N}m_N(T) < \infty$. We can rewrite this as
\begin{equation}
\label{eq:estimate-second-moment-beta}
\E [\beta_N^2(T)] \le C_{\beta}N^2.
\end{equation}

\subsection{Proof of Lemma~\ref{lemma:mgle-to-0}} Consider the size of each jump of the process $\bigl(\mu^{(N)}, f\bigr)$. At the emergence of a new bank with reserves $y$ at time $t$, the empirical measure process jumps
\begin{equation}
\label{eq:jump-birth}
\mbox{from}\quad \mu^{(N)}_{t-} = \frac1{\mathcal N_N(t-)}\sum\limits_{i=1}^{\mathcal N_N(t-)}\delta_{X^{(N)}_i(t-)}\quad \mbox{to}\quad \mu^{(N)}_t = \frac1{\mathcal N_N(t-)+1}\Biggl[\sum\limits_{i=1}^{\mathcal N_N(t-)}\delta_{X^{(N)}_i(t-)} + \delta_y\Biggr].
\end{equation}
Therefore, the displacement of $(\mu^{(N)}, f)$ is equal to
$$
\frac1{\mathcal N_N(t-)+1}\Biggl[\sum\limits_{i=1}^{\mathcal N_N(t-)}f\bigl(X_i^{(N)}(t-)\bigr) + f(y)\Biggr] - \frac1{\mathcal N_N(t-)}\sum\limits_{i=1}^{\mathcal N_N(t-)}f\bigl(X_i^{(N)}(t-)\bigr).
$$
This random variable  is dominated a.s.~by $2\norm{f}/\mathcal N_N(t-)$. Similarly, at the default of the $i$th bank (assume without loss of generality that $i = 1$), the displacement in $(\mu^{(N)}, f)$ is
\begin{align}
\label{eq:jump-default}
\begin{split}
\frac1{\mathcal N_N(t-)-1}&\sum\limits_{j=2}^{\mathcal N_N(t-)}f((1 - z_j)X^{(N)}_j(t-)) -  \frac1{\mathcal N_N(t-)}\sum\limits_{i=1}^{\mathcal N_N(t-)}f(X^{(N)}_i(t-)),
\\ z_j & \sim \cD_{\mathcal N_N(t-), \mathcal S_N(t-),X_1^{(N)}(t-)}.
\end{split}
\end{align}
The expression in \eqref{eq:jump-default} is dominated by $2(\norm{D_1f} + \norm{f})/\mathcal N_N(t-)$. To conclude, in both cases, recalling the definition of $\triplenorm{\cdot}$ in \eqref{eq:triplenorm}, the displacement of $(\mu^{(N)}, f)$ is dominated by
\begin{equation}
\label{eq:size-jump}
\frac{2}{\mathcal N_N(t-)}\triplenorm{f}.
\end{equation}
Next, the quadratic variation $\langle \mathcal M_N^f\rangle$ satisfies
\begin{equation}
\label{eq:quadratic-variation}
\mathrm{d}\,\langle \mathcal M_N^f\rangle_t = \frac{\sigma^2}{\mathcal N_N(t)}\bigl(\mu^{(N)}_t, (D_1f)^2\bigr)\,\mathrm{d}t.
\end{equation}
From~\eqref{eq:quadratic-variation}, it follows that
\begin{equation}
\label{eq:new-inequality}
\langle \mathcal M^f_N\rangle_T \le \frac{\sigma^2}{\mathfrak M^-_N(T)}\int_0^T\left(\mu^{(N)}_s, (D_1f)^2\right)\,\mathrm{d}s \le \frac{\sigma^2T\cdot\norm{D_1f}^2}{\mathfrak M_N^-(T)}.
\end{equation}
On the time intervals when there are no banks at all, with $\mathcal N_N(t) = 0$, the martingale $\mathcal M_N^f$ stays in fact constant, therefore we can neglect these intervals in our calculations. Apply~\eqref{eq:r-moment} with $r/2$ instead of $r$ to get:
\begin{equation}
\label{eq:qv-to-0}
\E \left[ \langle \mathcal M^f_N\rangle_T  \right] \to 0.
\end{equation}
Next, from Lemma~\ref{lemma:preliminary-estimate} we get that $(\mathcal M^f_N)_{N \ge 1}$ is uniformly a.s.~bounded on $[0, T]$ (by a constant $CT\cdot \triplenorm{f} + 2\norm{f}$). Extract a subsequence $(\mathcal M^f_{N_j}(T))_{j \ge 1}$ which converges a.s.~and (by Lebesgue dominated convergence theorem) in $L^2$ to a random variable $\xi$. Let $\mathcal M_{\infty}^f(t) := \E(\xi\mid\mathfrak F_t)$. Then by the standard martingale inequality
$$
\E\sup\limits_{0 \le t \le T}\bigl(\mathcal M_{\infty}^f(t) - \mathcal M_{N_j}^f\bigr)^2 \le 4\E\bigl(\xi - \mathcal M_{N_j}^f\bigr)^2 \to 0.
$$
Therefore, we can extract a subsequence such that
$$
\mathcal M^f_{N_j'} \to \mathcal M_{\infty}^f\quad \mbox{a.s.~uniformly on}\quad [0, T].
$$
From~\eqref{eq:size-jump}, combined with estimates from below in subsection 7.4, we conclude that the process $\mathcal M_{\infty}^f$ is a.s.~continuous. Moreover, it has zero quadratic variation by~\eqref{eq:qv-to-0}. Any continuous martingale with zero quadratic variation is constant. Therefore, $\mathcal M_{\infty}^f(t) = \mathcal M_{\infty}^f(0) = 0$. Finally, every subsequence $(\mathcal M_{N}^f)_{N \ge 1}$ contains its own subsequence which converges to $0$ uniformly in $L^2$. The result of Lemma~\ref{lemma:mgle-to-0} immediately follows from here.

\subsection{Proof of Lemma~\ref{lemma:bdd-moment}} Recall that
$$
\bigl(\mu^{(N)}_t, f_p\bigr) = \frac1{\mathcal N_N(t)}\sum\limits_{i=1}^{\mathcal N_N(t)}\left[X_i^{(N)}(t)\right]^p \qquad \text{for }\quad \mathcal N_N(t) \ge 1.
$$
If $\mathcal N_N(t) = 0$, then $\bigl(\mu^{(N)}_t, f_p\bigr) = 0$. Therefore,
\begin{equation}
\label{eq:1140}
\sup\limits_{0 \le t \le T}\bigl(\mu^{(N)}_t, f_p\bigr) \le \frac1{\mathfrak M^-_N(T)\vee 1}\sum\limits_{i=1}^{\mathfrak M^+_N(T)}\sup\limits_{t \le T}\left[X_i^{(N)}(t)\right]^p.
\end{equation}
The supremum inside the sum in the right-hand side of~\eqref{eq:1140} is taken over all $t \in [0, T]$ such that $X_i^{(N)}(t)$ is well-defined; that is, the $i$th bank exists at time $t$. Recall that $\beta_N(T)$ is defined as a pure birth process in Section 7.4. Use for~\eqref{eq:1140} the estimate $\mathfrak M^+_N(T) \le \beta_N(T)$, Wald's identity and the estimate \eqref{eq:r-moment} for $\mathfrak M^-_N(T)$ with $r = 2$. We get:
\begin{align}
\label{eq:final-answer}
\E    \left[\sup\limits_{0 \le t \le T}\bigl(\mu^{(N)}_t, f_p\bigr) \right]^2 \le \E\left[\mathfrak M^-_N(T)\vee 1\right]^{-2}\cdot \E \left[\sum\limits_{k=1}^{\beta_N(T)}\sup\limits_{t \le T}\left[X_i^{(N)}(t)\right]^p\right]^2.
\end{align}
The second multiple in the right-hand side of~\eqref{eq:final-answer} is stochastically dominated by the random sum of random variables
\begin{equation}
\label{eq:random-sum}
\sum\limits_{i=1}^{\beta_N(T)}\xi^p_i,\quad \xi_i := \eta_i\exp\left(\sup\limits_{0 \le t \le T}\left[rt + \sigma W_i(t)\right]\right).
\end{equation}
Here, $\eta_i\sim \nu$ are i.i.d.~random variables, $\nu$ is a probability measure in $\mathcal P_p$ which stochastically dominates each $\mu^{(N)}_0$ and $\mathcal B_{\infty, N}$. Such measure exists because these measures have uniformly bounded $p$th moment. This, in turn, follows from $\mu^{(N)}_0 \to \mu^{(\infty)}_0$ in $\mathcal W_p$ (this is an assumption of Theorem~\ref{thm:hydro}) and Assumption~\ref{asmp:birth-measures}. Finally, $W_1, W_2, \ldots$ are i.i.d.~Brownian motions, independent of $\eta_i$, and the birth process $\beta$ is  independent of these Brownian motions and of $\eta_i$. By Wald's identity, we get for some constant $C_1$:
\begin{equation}
\label{eq:expectation-of-random-sum}
\E \Big[ \sum\limits_{i=1}^{\beta_N(T)}\xi_i^p \Big] = \E[\beta_N(T)]\cdot \E[\xi_1^p] \le C_1N.
\end{equation}
Combining~\eqref{eq:random-sum} with~\eqref{eq:expectation-of-random-sum}, we get for some constant $C_2$:
\begin{equation}
\label{eq:1152}
\E \Big[ \sum\limits_{i=1}^{\beta_N(T)}\sup\limits_{t}\left\{X_i^{(N)}(t)\right\}^p  \Big] \le C_2N.
\end{equation}
Next, the variance of this random sum~\eqref{eq:random-sum} is equal to
\begin{equation}
\Var\beta_N(T)\cdot \E\xi_1^{2p} + \E\beta_N(T)\cdot\Var\xi_1^p \le C_3N^2.
\label{eq:variance-of-random-sum}
\end{equation}
Here we used the estimate~\eqref{eq:estimate-second-moment-beta}. Combining~\eqref{eq:expectation-of-random-sum} and~\eqref{eq:variance-of-random-sum}, we get the following estimate: For some constant $C_4$,
\begin{equation}
\label{eq:second-moment-of-random-sum}
\E\beta_N^2(T) \le C_4N^2.
\end{equation}
In turn, combining~\eqref{eq:final-answer},~\eqref{eq:random-sum}, we complete the proof.

\subsection{Proof of Lemma~\ref{lemma:tight}} Recall $C_T$ from Lemma~\ref{lemma:bdd-moment}. Take any $\eta > 0$, and let $C := C_T/\eta$. Consider the subset $\mathcal K := \{\nu \in \mathcal P_q\mid (\nu, f_p) \le C\}$, which is compact in $\mathcal P_{q}$ by Lemma~\ref{lemma:precompact}.  From the standard Markov inequality, we have:
$$
\PP\left[\mu_t^{(N)} \in \mathcal K\quad \forall\, t \in [0, T]\right] > 1 - \eta.
$$
Next, take the algebra $\mathfrak A$ in $C_b(\mathcal P_q)$ generated by $\mathfrak M := \{(\cdot, f)\mid f \in \mathcal C^2_b\}$. This set $\mathfrak M$ separates points: for every $\nu'$ and $\nu''$ in $\mathcal P_q$, there exists an $f \in \mathcal C^2_b$ such that $(\nu', f) \ne (\nu'', f)$. This set $\mathfrak M$ also contains $1$, because $f_0 = 1 \in \mathcal C^2_b$. By the Stone-Weierstrass theorem \cite[Section 4.7]{Folland}, the algebra $\mathfrak A$ is dense in $C_b(\mathcal P_q)$ in the topology of uniform convergence on compact subsets.

\smallskip

From Lemmas~\ref{lemma:mgle-to-0},~\ref{lemma:preliminary-estimate}, the sequence $((\mu^{(N)}_t, f),\, t \in [0, T])_{N \ge 1}$ is tight in $\cD[0, T]$ for every $f \in \mathcal C^2_b$. Since $(\mu^{(N)}_t, f)$ is uniformly bounded by $\norm{f}$, for every collection $g_1, \ldots, g_m \in \mathcal C^2_b$ the following sequence is tight in $\cD[0, T]$:
$$
\bigl(\mu^{(N)}_t, g_1\bigr)\bigl(\mu^{(N)}_t, g_2\bigr)\cdot\ldots\cdot\bigl(\mu^{(N)}_t, g_m\bigr).
$$
Therefore, for every $\Phi \in \mathfrak A$, the following sequence is tight in $\cD[0, T]$: $\bigl(\Phi\bigl(\mu^{(N)}_t\bigr),\, t \in [0, T]\bigr)_{N \ge 1}$. Apply the criteria of relative compactness: \cite[Proposition 3.9.1]{EthierKurtz}, and complete the proof.

\section{Proof of Theorem~\ref{thm:hydro-new}}\label{sec:proof-new}

\subsection{Overview of the proof}  The proof is similar to the proof of Theorem~\ref{thm:hydro}, except the following changes. We cannot apply Lemma~\ref{lemma:fundamental-convergence} directly, because the birth intensities and the default contagion measures are scaled according to the {\it initial} number of banks $\mathcal N_N(0) = N$, rather than the {\it current} one $\mathcal N_N(t)$. Therefore, we need to take into account the ratio $N^{-1}\mathcal N_N(t)$, and its limit as $N \to \infty$ is $\mathcal N_{\infty}(t)$.

\begin{lemma}
For every $q > 0$, we have the following estimates:
\begin{align}
\label{eq:estimate-moments}
\begin{split}
&\sup\limits_{N \ge 1}\mathbb E\Bigl[N^{-1}\max\limits_{0 \le t \le T}\mathcal N_N(t) \Bigr]^q < \infty;\\
&\sup\limits_{N \ge 1}\mathbb E\Bigl[N^{-1}\max\limits_{0 \le t \le T}S_N(t) \Bigr]^q < \infty.
\end{split}
\end{align}
\label{lemma:basic-moment-estimates}
\end{lemma}

\begin{lemma} The sequence $(N^{-1}\mathcal N_N(t),\, 0 \le t \le T)$ of processes in $\cD[0, T]$ is tight. 
\label{lemma:normalized-counter}
\end{lemma}

From Lemma~\ref{lemma:basic-moment-estimates}, we prove the statement of Lemma~\ref{lemma:tight}: the sequence
$$
(\mu^{(N)})_{N \ge 1}\quad \mbox{is tight in}\quad \cD([0, T], \mathcal W_q).
$$
Next, take a weak limit point $\mathcal N_{\infty} = (\mathcal N_{\infty}(t),\, 0 \le t \le T)$ from Lemma~\ref{lemma:normalized-counter}, and a weak limit point $\tilde\mu^{(\infty)}$ of $(\mu^{(N)})_{N \ge 1}$ in $\cD([0, T], \mathcal W_q)$, for some $q \in (1, p)$. Denote by $\tm(t)$ the mean of $\tilde \mu^{(\infty)}(t)$. The functional $\nu \mapsto (\nu, f_1)$ is continuous in $\mathcal W_q$ for $q > 1$. Therefore, taking a limit as $N \to \infty$, we get that the following process $\tilde{N}$ is a martingale:
$$
\tilde{N}(t) := \mathcal N_{\infty}(t) - \int_0^t\bigl[\lambda_{\infty}(\tm(s)) - \mathcal N_{\infty}(s)\left(\tilde\mu_{\infty}(s), \kappa_{\infty}(\tm(s), \cdot)\right)\bigr]\,\mathrm{d}s.
$$
It is continuous, and has zero quadratic variation; therefore, $\tilde{N}$ is constant (equal to its initial value $\tilde{N}(0) =1$). Thus $\mathcal N_{\infty}$ is, in fact, a deterministic function satisfying \eqref{eq:ratio-limit}. Finally, let us adjust Lemma~\ref{lemma:fundamental-convergence}, so that the expression converges to the right type of the generator. 

\begin{lemma} Take a function $f \in \mathcal C^2_b$. Consider a sequence $(\mathbf{x}^{(k)})_{k \ge 1}$ in $\mathcal X$ with
\begin{equation}
\label{eq:standard-sequence-new}
\frac{\mathbf n\left(\mathbf{x}^{(k)}\right)}{k} \to n_{\infty}\quad\mbox{and}\quad \mu_{\mathbf{x}^{(k)}} \to \nu\quad \mbox{in}\quad \mathcal W_p.
\end{equation}
For $\tilde{\mathcal A}$ defined in~\eqref{eq:new-operator}, we have: $\mathfrak{L}\mathcal E_f\left(\mathbf{x}^{(k)}\right) \to \tilde{\mathcal A}(n_{\infty}, \nu, f)$ as $k \to \infty$. 
\label{lemma:fundamental-convergence-new}
\end{lemma}

From Lemma~\ref{lemma:fundamental-convergence-new}, we get that every weak limit point 
$$
(\mathcal N_{\infty}(t),\, \mu^{(\infty)}_t,\, 0 \le t \le T)\quad \mbox{of}\quad (N^{-1}\mathcal N_N(t),\, \mu^{(N)}_t,\, 0 \le t \le T)\quad \mbox{in}\quad D([0, T], \mathbb R\times\mathcal P_p)
$$
satisfies the system~\eqref{eq:new-generator},~\eqref{eq:ratio-limit}. By uniqueness from Remark~\ref{rmk:existence-uniqueness}, we complete the proof.

\subsection{Proof of Lemma~\ref{lemma:basic-moment-estimates}} The estimation of the number of banks from above and below remains the same as in Lemma~\ref{lemma:bdd-moment}: In the proof of the upper estimate, we now have the intensity of births from level $n$ to level $n+1$ for the benchmark process $\beta_N$ (now dependent on $N$) equal to $C_{\lambda}N$, independent of $n$. Therefore,
$$
N^{-1}\beta_N(t) = 1 + N^{-1}\theta_N,\ \theta_N \sim \Poi(C_{\lambda}N).
$$
Applying the law of large numbers to $N^{-1}\theta_N$ and observing that convergence holds in every space $L^q$, we prove the first formula in~\eqref{eq:estimate-moments}. Let us show the second formula: 
\begin{align}
\label{eq:1422}
\begin{split}
\left[N^{-1}\mathcal S_N(t)\right]^q & = \left[N^{-1}\mathcal N_N(t)\right]^q\left[\mathcal N^{-1}_N(t)\mathcal S_N(t)\right]^q =  \left[N^{-1}\mathcal N_N(t)\right]^q\cdot \bigl(\mu^{(N)}_t, f_1\bigr)^q \\ & \le \left[N^{-1}\mathcal N_N(t)\right]^q \cdot \bigl(\mu^{(N)}_t, f_q\bigr).
\end{split}
\end{align}
In the last step of~\eqref{eq:1422}, we applied the inequality $(\mathbb E \xi)^q \le \mathbb E \xi^q$ for the random variable $f_1$ integrated against the probability measure $\mu^{(N)}_t$. Taking the supremum of~\eqref{eq:1422} and applying expected value, by the Cauchy-Schwarz inequality,
\begin{align} \notag
\left[\mathbb E\sup\limits_{0 \le t \le T}\left[N^{-1}\mathcal S_N(t)\right]^{q}\right]^2 & \le \mathbb E\sup\limits_{0 \le t \le T}\left[N^{-1}\mathcal N_N(t)\right]^{2q}\cdot\mathbb E \sup\limits_{0 \le t \le T}\bigl(\mu^{(N)}_t, f_q\bigr)^2 \\
& \le \mathbb E\sup\limits_{0 \le t \le T}\left[N^{-1}\mathcal N_N(t)\right]^{2q}\cdot\mathbb E\sup\limits_{0 \le t \le T}\bigl(\mu^{(N)}_t, f_{2q}\bigr),
\label{eq:1423}
\end{align}
where we use the inequality $(\mathbb E\xi)^2 \le \mathbb E \xi^2$ with $\xi = f_q$ and the probability measure $\mu^{(N)}_t$ in the second inequality. Finally, in ~\eqref{eq:1423} we may apply the second estimate in~\eqref{eq:estimate-moments} to the first term in the right-hand side, and estimate the second term similarly to Lemma~\ref{lemma:bdd-moment}:
$$
\sup\limits_{N \ge 1}\mathbb E\sup\limits_{0 \le t \le T}\bigl(\mu^{(N)}_t, f_{2q}\bigr) < \infty.
$$
This completes the proof that the right-hand side of~\eqref{eq:1423} is bounded from above by a constant, independent of $N$.

\subsection{Proof of Lemma~\ref{lemma:normalized-counter}}
The $N$th process starts from $1$, jumps upward by $N^{-1}$ with intensity
\begin{equation}
\label{eq:birth-intensity-estimate}
\lambda_{N}(S_N(t)) \le C_{\lambda}(N + S_N(t)) = NC_{\lambda}(1 + N^{-1}S_N(t)),
\end{equation}
and downward by $-N^{-1}$ with intensity
\begin{equation}
\label{eq:default-intensity-estimate}
\sum\limits_{i=1}^{\mathcal N_N(t)}\kappa_N\bigl(N^{-1}S_N(t), X_i^{(N)}(t)\bigr) \le C_{\kappa}\mathcal N_N(t).
\end{equation}
These estimates in~\eqref{eq:birth-intensity-estimate} and~\eqref{eq:default-intensity-estimate} are taken from Assumptions~\ref{asmp:birth-intensities-new} and~\ref{asmp:default-intensities-new}, respectively. By Lemma~\ref{lemma:basic-moment-estimates}, there exists a constant $C > 0$ such that the intensities of jumps of $N^{-1}\mathcal N_N(\cdot)$ are bounded (in $L^q$ for every $q > 0$) by $CN$, and the size of jumps is equal to $N^{-1}$. Therefore,
\begin{equation}
\label{eq:mgle-counter}
\frac1N\mathcal N_N(t) - \frac1N\int_0^t\Bigl[\lambda_N(N^{-1}S_N(s)) - \sum\limits_{i=1}^{\mathcal N_N(s)}\kappa_N\bigl(N^{-1}S_N(s), X_i^{(N)}(s)\bigr)\Bigr]\,\mathrm{d}s,\, 0 \le t \le T,
\end{equation}
is a local martingale, and because it is in $L^p$ an actual martingale. 
Similarly to Lemma~\ref{lemma:mgle-to-0}, we can imply that the sequence~\eqref{eq:mgle-counter} converges to $0$. From Lemma~\ref{lemma:basic-moment-estimates} we get that for some constant $C$, for all $s, t \in [0, T]$ and $N \ge 1$, we get: $\E(N^{-1}\mathcal N_N(t) - N^{-1}\mathcal N_N(s))^2 \le C(t-s)^2$, 
which implies tightness by \cite[Chapter 2, Problem 4.11]{KSBook}.

\subsection{Proof of Lemma~\ref{lemma:fundamental-convergence-new}}
By Lemma~\ref{lemma:technical-convergence}, $\overline{\mathbf{x}}^{(N)} \to \overline{\nu}$. The rest of the proof is similar to that of Lemma~\ref{lemma:fundamental-convergence}, but with the following changes. As $N \to \infty$, $\left[\mathfrak{n}\bigl(\mathbf{x}^{(N)}\bigr)\right]^{-1}\lambda_N(\overline{\bx})  \to \mathcal N_{\infty}^{-1}\lambda_{\infty}\bigl(\overline{\nu}\bigr)$. Therefore, instead of~\eqref{eq:second-conv}, we have:
\begin{equation}
\label{eq:second-conv-new}
I_2 \to n_{\infty}^{-1}\lambda_{\infty}\left(\overline{\nu}\right)\left[\left(\mathcal B_{\infty, \overline{\nu}}, f\right)  - \left(\nu, f\right)\right].
\end{equation}
A similar difference between Assumptions~\ref{asmp:default-measures} and~\ref{asmp:default-measures-new} means that, instead of~\eqref{eq:I31-conv}, we have:
\begin{equation}
\label{eq:I31-conv-new}
I_{3, 1} \to -n_{\infty}\left(\nu, \kappa_{\infty}(\overline{\nu}, \cdot)\cdot\overline{\cD}_{\infty}\bigl(\mathbf{n}\bigl(\mathbf{x}^{(N)}\bigr)\overline{\nu}, \cdot\bigr)\right)\left(\nu, D_1f\right).
\end{equation}
Convergence statements~\eqref{eq:first-conv} and~\eqref{eq:two-conv} stay the same. This completes the proof of Lemma~\ref{lemma:fundamental-convergence-new}.  

\section{Proof of Theorem~\ref{thm:individual}}\label{sec:theorem42}

\subsection{Overview of the proof} This is similar to the proof of Theorem~\ref{thm:hydro}, but easier, since we deal with real-valued processes instead of measure-valued ones. Let us split this proof into lemmas. For every function $f : (0, \infty) \to \mathbb R$, we can define a corresponding function $\varphi_f : \mathcal X \to \mathbb R$ as follows:
$$
\varphi_f(\mathbf{x}) \mapsto f(x_1),\quad \mathbf{x} \ne \varnothing;\quad \varphi_f(\varnothing) := 0.
$$
This function $\varphi_f$ effectively depends only on $x_1$. The generator $\mathfrak{L}$ from~\eqref{eq:generator} applied to $\varphi_f$ gives 
\begin{align}
\label{eq:generator-on-x-1}
\begin{split}
\mathfrak{L}\varphi_f&\left(\mathbf{x}\right)  = \mathcal G f(x_1) - \kappa_{\mathfrak n(\bx)}(\mathfrak{s}(\bx), x_1)f(x_1)\\ & + \sum\limits_{i=2}^{\mathfrak n(\bx)}\kappa_{\mathfrak n(\bx)}(\mathfrak s(\bx), x_i)\int_0^{\infty}\left[f(x_1(1-z)) - f(x_1)\right]\,\cD_{\mathfrak n(\bx), \mathfrak s(\bx),x_i}(\mathrm{d} z).
\end{split}
\end{align}
By It\^o's formula:
\begin{align}
\label{eq:Ito-chaos}
\begin{split}
f(X_1^{(N)}(t)) = f(X_1^{(N)}(0)) + \int_0^t\mathfrak{L}\varphi_f\left(X^{(N)}(s)\right)\,\mathrm{d}s + \widehat{\mathcal M}^f_N(t).
\end{split}
\end{align}
Here we denote by $\widehat{\mathcal M}_N^f =  (\widehat{\mathcal M}_N^f(t),\, t \ge 0)$ a local martingale. Its trajectories are right-continuous with left limits. Between jumps, it behaves according to the following stochastic equation:
\begin{equation}
\label{eq:M-martingale-x-1}
\mathrm{d}\,\widehat{\mathcal M}_N^f(t) := \sigma\,(D_1f)\bigl(X^{(N)}_1(s)\bigr)\,\mathrm{d}W_1(s),\quad t \ge 0,
\end{equation}
The following two lemmas are proved similarly to Lemmas~\ref{lemma:fundamental-convergence},~\ref{lemma:preliminary-estimate}.

\begin{lemma} Take a sequence $(\mathbf{x}^{(k)})_{k \ge 1}$ as in~\eqref{eq:standard-sequence} with $x^{(k)}_1 \to x_1^{(\infty)}$ as $k \to \infty$. For $f \in \mathcal C^2_b$, we get: $\mathfrak{L}\varphi_f\left(\mathbf{x}^{(k)}\right) \to \mathcal A^*( \overline{\nu}, f)$, where $\mathcal A^*$ is defined in~\eqref{eq:A-star}.
\label{lemma:convergence-x-1}
\end{lemma}

\begin{lemma} For a constant $C_*$ and all $f \in \mathcal C^2_b$, $\mathbf{x} \in \mathcal X\setminus\{\varnothing\}$, we have: $\left|\mathfrak{L}\varphi_f(\mathbf{x})\right| \le C_*\triplenorm{f}$.
\label{lemma:213}
\end{lemma}

Next, let us state some new lemmas.

\begin{lemma} For some constant $C_{T, q} > 0$, we get:
\begin{equation}
\label{eq:universal-bound}
\E \left[\varphi_{f_q}\left(X^{(N)}(t)\right) \right] = \E \left[\bigl(X^{(N)}_1(t)\bigr)^q \right] \le C_{T, q},\quad t \in [0, T].
\end{equation}
\label{lemma:upper-bd}
\end{lemma}

\begin{lemma} \label{lemma:mgle-tight} For $f \in \mathcal C^2_b$, the sequences $(\widehat{\mathcal M}_N^f)_{N \ge 1}$ and $(X_1^{(N)})_{N \ge 1}$  are tight in $\cD[0, T]$.
%
\end{lemma}

Extract a convergent subsequence $X_1^{(N_j)} \Rightarrow X_1^{(\infty)}\quad\mbox{in}\quad \cD[0, T]$. From Theorem~\ref{thm:hydro}, Lemmata~\ref{lemma:convergence-x-1},~\ref{lemma:213}, we conclude that for every $f \in \mathcal C^2_b$ (with the usual convention that $f(\Delta) = 0$ at the cemetery state), the following process is a local martingale:
\begin{align*}
f\bigl(X_1^{(\infty)}(t)\bigr)  - f\bigl(X_1^{(\infty)}(0)\bigr) - \int_0^t\cG_{\mathbf{m}(s)}f\bigl(X^{(\infty)}_1(s)\bigr) \mathrm{d}s - \kappa_{\infty}\bigl(\mathbf{m}(t), X^{(\infty)}_1(t)\bigr)f\bigl(X_1^{(\infty)}(t)\bigr).
\end{align*}
By uniqueness of the martingale problem for geometric (killed) Brownian motion, this completes the proof of Theorem~\ref{thm:individual}.

\subsection{Proof of Lemma~\ref{lemma:upper-bd}} The process $X^{(N)}_1$ can only jump down. As long as it does not jump, it behaves as a geometric Brownian motion. Thus, $
X^{(N)}_1(t) \le  X^{(N)}_1(0)\exp[(r - \sigma^2/2)t + \sigma W_1(t)]$.
Fix a $q \in (0, p]$. Take the expectation of the $q$th degree of the maximum of $X^{(N)}_1(t)$ over $t \in [0, T]$. Analogous to Lemma~\ref{lemma:bdd-moment}, we get~\eqref{eq:universal-bound}.

\subsection{Proof of Lemma~\ref{lemma:mgle-tight}} For any function $f \in \mathcal C^2_b$, the process $f\bigl(X_1^{(N)}(t)\bigr)$ (until its killing time) is represented as in~\eqref{eq:Ito-chaos}. The local martingale $\widehat{\mathcal M}_N^f$ has quadratic variation $\langle\widehat{\mathcal M}_N^f\rangle$ with
$$
\frac{d\langle \widehat{\mathcal M}_N^f\rangle_t}{\mathrm{d}t} \le \sigma^2\norm{D_1f} < \infty.
$$
The intensity of jumps of $\widehat{\mathcal M}_N^f$ at time $t$ can be estimated from Assumption~\ref{asmp:default-intensities}:
\begin{equation}
\label{eq:varkappa}
\sum\limits_{i=2}^{\mathcal N_N(t)}\kappa_{\mathcal N_N(t)}\bigl(S_N(t), X_i^{(N)}(t)\bigr) \le \mathcal N_N(t)\,C_{\kappa}.
\end{equation}
The displacement due to a default of $X_i$ at time $t$ is equal to
$$
\eta_i := f(X_1^{(N)}(t-)(1-\xi_i)) - f(X_1^{(N)}(t-)),\quad \xi_i \sim \cD_{\mathcal N_N(t-), S_N(t-),X_i^{(N)}(t-)}.
$$
Because $\norm{D_1f}$ is a well-defined finite quantity for $f \in \mathcal C^2_b$, this displacement $\eta_i$ can be estimated from above as $\norm{D_1f}\xi_i \le \norm{D_1f}$. Combining Assumption~\ref{asmp:default-measures} with this estimate, we get that the maximum size of jumps of $\widehat{\mathcal M}_N^f$ tends to zero in $L^p$, as $N \to \infty$. For $f \in \mathcal C_b^2$, the functions $f, D_1f, D_2f$ have finite norm $\norm{\cdot}$. From the representation~\eqref{eq:Ito-chaos}, we get:
$ \sup_{N \ge 1}\E [\widehat{\mathcal M}_N^f(T) ]^2 < \infty$.
Therefore, similarly to the proof of Lemma~\ref{lemma:mgle-to-0} in Section 7, we show that the sequence $(\widehat{\mathcal M}_N^f)_{N \ge 1}$ is tight in $\mathcal D[0, T]$. Lemma~\ref{lemma:upper-bd}, together with the Markov inequality, implies compact containment condition: for every $\varepsilon > 0$ and $T > 0$, there exists a compact set $\mathcal K \subseteq (0, \infty)$ such that
\begin{equation}
\label{eq:cpt-cond}
\PP\bigl(X^{(N)}_1(t) \in \mathcal K\quad\forall\, t \in [0, T]\bigr) > 1 - \varepsilon.
\end{equation}
This, together with \cite[Proposition 3.9.1]{EthierKurtz}, Lemma~\ref{lemma:convergence-x-1},~\ref{lemma:213},  tightness of $(\mathcal M^f_N)$, and convergence of initial conditions, proves tightness of $f(X_1^{(N)})$ in $\cD[0, T]$ for every $T > 0$.

\section{Appendix: System Construction} \label{sec: Appendix}

We define $(X, I, M)$ inductively, and the corresponding generator $\mathfrak L$ in (\ref{eq:generator}) above.  The initial conditions are defined as follows:
$$
X(0) := \mathbf{x}_0,\quad M(0) = N_0 := \mathfrak n(\mathbf{x}_0),\quad I(0) := \{1, \ldots, N_0\}.
$$
Assume we already defined the system $(X(t), I(t), M(t))$ for $t \le \tau_k$, where $k = 0, 1, 2, \ldots$ is given. Let us define it on $[\tau_k, \tau_{k+1})$. First, assume $N_{k } := N(\tau_k) \ge 1$ with $I(\tau_k) \ne  \varnothing$. Define auxiliary stochastic processes $X^*_{k,i} = (X^*_{k,i}(s),\, s \ge 0)$ for $i \in I(\tau_k)$ to be independent geometric Brownian motions with drift $\mu  $ and  diffusion $\sigma^{2}$, and with initial value $X_{i}(\tau_{k})$. Define stopping times $\tau_{k,i}$:
\begin{align}
X^*_{k,i}(s) & = X_i(\tau_k)\exp\Bigl[\Bigl(r - \frac{\sigma^2}2\Bigr)t + \sigma W_{k,i}(s)\Bigr];\quad S^*_k(u) := \sum\limits_{i \in I(\tau_k)}X^*_{k,i}(s),\ \ s \ge 0;\\
\tau_{k,i} & := \inf \Bigl\{s \ge 0 :  \int_0^s\kappa_{N_k}(S^*(u), X^*_{k,i}(u))\,\md u \ge \eta_{k,i} \Bigr\},\ i \in I(\tau_k);\\
\tau_{k,0} & := \inf\Bigl\{s \ge 0 :  \int_0^s\lambda_{N_k}(S^*(u))\,\md u \ge \eta_{k,0}\Bigr\},
\end{align}
given the killing rate $\kappa_n(s,x) $ and birth rate $\lambda_n(s) $ functions for $n \in \mathbb N_{0}$, $x \in \mathcal X$, $s \in \mathbb R_{+}$. Here $\tau_{k,0}$ represents the necessary inter-arrival random time for the potential birth, and $\tau_{k,i}$ represents the potential default of bank $i$. The next event is now determined almost surely uniquely by the minimal arrival $ \min \{\tau_{k,i},\, i \in I(\tau_k)\cup\{0\}\} $ of these potential events. We set $\tau_{k+1} := \tau_k + \tau_{k,j}$ with the index $j := \arg\min_{i \in I (\tau_{k}) \cup \{0\}} \tau_{k,i}$, and define for $t \in [\tau_k, \tau_{k+1})$:
$$
X_i(t) := X^*_{k,i}(t - \tau_k),\, i \in I(\tau_k);\qquad I(t) := I(\tau_k),\quad M(t) := M(\tau_k).
$$
Then we consider two cases. If $j = 0$, a new bank emerges at time $\tau_{k+1}$ and set
\begin{align*}
M(\tau_{k+1}) & := M(\tau_k) + 1,\quad I(\tau_{k+1}) := I(\tau_k)\cup\{M(\tau_{k+1})\}; \\
X_i(\tau_{k+1}) & := X_i(\tau_{k+1}-),\, i \in I(\tau_k),\quad X_{M(\tau_{k+1})} :=\zeta_{k,M(\tau_k), S(\tau_{k+1}-)}.
\end{align*}
If $j \in I(\tau_{k+1})$, the $j$-th bank defaults at time $\tau_{k+1}$ with $X_{j}(\tau_{k+1}) := \varnothing $ and
\begin{align*}
M(\tau_{k+1}) & := M(\tau_k),\quad I(\tau_{k+1}) := I(\tau_k)\setminus\{j\};
\\
X_i(\tau_{k+1}) & := X_i(\tau_{k+1}-)\left[1 - \xi_{i,j,N_k, S(\tau_{k+1}-),X_j(\tau_{k+1}-)}\right],\, i \in I(\tau_{k+1}).
\end{align*}
Second, for the case of $N(\tau_k) = N_k = 0$ we have no banks at time $\tau_k$, i.e., $I(\tau_k) = \varnothing$. In that case the system regenerates via a birth.  Let us set $\tau_{k+1} := \tau_k + (\eta_{k,0}/ \lambda_{0}(0) ) 
$, and
\begin{align*}
N(t) & :=  0, \quad I(t) :=\varnothing, \quad X(t) := \varnothing,\quad t \in [\tau_k, \tau_{k+1});
\\
M(\tau_{k+1}) &:= M(\tau_k) + 1;\quad I(\tau_{k+1}) := \{M(\tau_{k+1})\},\quad X_{M(\tau_{k+1})}(\tau_{k+1}) := \zeta_{k,M(\tau_{k+1}),0}.
\end{align*}
The triple $(X, I, M)$ is now well-defined with $\, \lvert I ( \cdot ) \rvert =  \mathfrak n ( X(\cdot)) \le M(\cdot)\,$ on the time interval $[0, \tau_{\infty})$, where $\tau_{\infty} := \lim_{k \to \infty}\tau_k$. By construction, this is a Markov process on the state space
\begin{equation}
\Xi := \{(\mathbf{x}, \mathfrak{i}, \mathfrak m) \in \mathcal X\times 2^{\mathbb N}\times \mathbb N \, : \,  |\mathfrak{i}| = \mathfrak n(\bx) \le \mathfrak m\} ,
\end{equation}
and its law is uniquely determined up to explosion time.

\section*{Acknowledgements}

Part of the research was supported by National Science Foundation under grants NSF DMS-1615229, NSF DMS-1521743, and NSF DMS-1409434. Sarantsev benefited from the discussion with Clayton Barnes, Ricardo Fernholz, and Mykhaylo Shkolnikov.

\medskip\noindent

\bibliographystyle{alpha} 

\bibliography{references}

\end{document}